\documentclass[draft,twoside,12pt,reqno]{amsart}
\usepackage{amssymb,latexsym,verbatim} 

\theoremstyle{plain}
\newtheorem{theorem}{Theorem}
\newtheorem{lemma}[theorem]{Lemma}

\theoremstyle{definition}

\theoremstyle{remark}

\renewcommand{\geq}{\geqslant}
\renewcommand{\leq}{\leqslant}

\begin{document}

\newcommand\AGammaL{\textup{A}\Gamma\textup{L}}
\newcommand\AGL{\textup{AGL}}
\newcommand{\Aut}{\textup{Aut}}
\newcommand\BigO{\textup{O}}
\newcommand{\BO}{\textup{BO}}
\newcommand{\BarHomo}{\,{}^{\overline{\hskip2.5mm}}\,}
\newcommand{\ch}{\textup{\;char\ }} 
\newcommand{\C}{\mathbb{C}}
\newcommand\Der{\textsc{Der}}
\newcommand{\diag}{\textup{diag}}
\newcommand{\e}{\varepsilon}
\newcommand{\E}{\mathbb{E}}
\newcommand{\End}{\textup{End}}
\newcommand{\eps}{\varepsilon}
\newcommand{\ext}{\sqsubset}
\newcommand{\F}{\mathbb{F}}
\newcommand\G{\mathbb{G}}
\newcommand\Gal{\textup{Gal}}
\newcommand\GammaL{\Gamma\textup{L}}
\newcommand\GL{\textup{GL}}
\newcommand\GO{\textup{GO}}
\newcommand\GSp{\textup{GSp}}
\newcommand\im{\textup{im}}
\newcommand\Inn{\textup{Inn}}
\newcommand\IsIsoTo{\lessapprox}
\newcommand\Norm{\textsc{Norm}}
\newcommand\normal{\trianglelefteq}
\newcommand\normalizes{\trianglerighteq}
\newcommand\onto{\twoheadrightarrow}
\newcommand\ord{\textup{ord}}
\newcommand\Out{\textup{Out}}
\newcommand\POmega{\textup{P}\Omega}
\newcommand\PSL{\textup{PSL}}
\newcommand\PSU{\textup{PSU}}
\newcommand{\Q}{\mathbb{Q}}
\newcommand\SL{\textup{SL}}
\newcommand\SO{\textup{SO}}
\newcommand{\Stab}{\textup{Stab}}
\newcommand{\Symp}{\textup{Sp}}
\newcommand{\subheading}[1]{\vskip1.5mm\noindent{\sc #1}}
\newcommand\T{{\textup{T}}}
\newcommand\tensor{\otimes}
\newcommand\W{\mathbb{W}}
\newcommand\Wr{\,\textup{wr}\,}
\newcommand{\VwedgeV}{\Lambda^2 V}
\newcommand{\Z}{\mathbb{Z}}

\newcommand{\w}{\wedge}
\newcommand\Y{\,\textsf{Y}\,}
\newcommand{\Hom}{\textup{Hom}}

\hyphenation{induced sub-fields irred-ucible}

\title[\tiny\upshape\rmfamily $p$-groups having a unique proper non-trivial
  characteristic subgroup]{}

\date{23 July 2010}

\begin{center}{\Large
$p$-groups having a unique proper non-trivial
  characteristic subgroup}
\end{center}

\bigskip

\begin{center}
{\large S.\,P. Glasby}\\
Department of Mathematics\\   
Central Washington University\\
WA 98926-7424, USA\\
{\tt http://www.cwu.edu/$\sim$glasbys/}

\smallskip

{\large P.\ P.\ P\'alfy}\\
Alfr\'ed R\'enyi Institute of Mathematics\\
1364 Budapest, Pf.\ 127, HUNGARY\\
{\tt http://www.renyi.hu/$\sim$ppp/}\\

\smallskip

{\large Csaba Schneider\footnote{Corresponding author. 
Email: csaba.schneider@gmail.com; Fax: +351 217 954 288}}\\
Centro de \'Algebra da Universidade de Lisboa\\
Av. Prof. Gama Pinto, 2\\
1649-003 Lisboa, PORTUGAL\\
{\tt  http://www.sztaki.hu/$\sim$schneider/}
\end{center}

\begin{abstract}
We consider the structure of finite $p$-groups $G$
having precisely three characteristic subgroups, namely $1$, $\Phi(G)$
and $G$. 
The structure of $G$ varies markedly
depending on whether $G$ has exponent $p$ or $p^2$, and, 
in both cases, the study of such groups raises deep problems in representation
theory. 
We present classification theorems for 3- and 4-generator groups, and
we also study the existence of such $r$-generator 
groups with exponent $p^2$ for various values of $r$.
The automorphism group induced on the Frattini quotient is, in various
cases, related to a maximal linear group in Aschbacher's classification scheme.
\end{abstract}

\maketitle
\centerline{\noindent 2010 Mathematics subject classification:
 20D15, 20C20, 20E15, 20F28}

\section{Introduction}

Taunt \cite{T55} considered groups having precisely three
characteristic subgroups. As such groups have a {\bf u}nique proper
non-trivial {\bf c}haracteristic {\bf s}ubgroup, he called these UCS-groups.
He gave necessary, but not sufficient, conditions for the direct power
of a UCS-group to be a UCS-group. Taunt discussed
solvable UCS-groups in \cite{T55}, and promised a forthcoming paper
describing the structure of UCS $p$-groups.
However, his article on UCS $p$-groups was never written. 
The present paper is devoted to the study of these groups.

In our experience UCS $p$-groups are rather elusive, and it is unlikely
that a general classification could be given. However, the study of these
groups does lead to the exploration of very interesting problems in 
representation theory, and some interesting, sometimes even surprising, 
theorems can be proved.

The main results of this paper can be summarized as follows.

\begin{theorem}\label{main}
\begin{enumerate}
Let $G$ be a non-abelian UCS $p$-group where $|G/\Phi(G)|=p^r$.
\item[(a)] If $r=2$, then $G$ belongs to a unique isomorphism class.
\item[(b)] If $r=3$, then $G$ belongs to a unique isomorphism class if $p=2$,
and to one of two distinct isomorphism classes if $p>2$.
\item[(c)] If $r=4$ and either $p=2$ or $G$ has exponent $p$,
then $G$ belongs to one of eight distinct isomorphism classes. 
\item[(d)] Suppose that $r=4$ and $G$ has exponent $p^2$. 
Then $p\neq 5$. Further, if
$p\equiv \pm 1\pmod 5$ then $G$ belongs to a unique isomorphism class; while
if $p\equiv\pm 2\pmod 5$ then $G$ belongs to one of two distinct isomorphism classes.
\item[(e)]
Let $p$ be an odd prime, and let $k$ be a positive integer. Then
there exist non-abelian exponent-$p^2$ UCS-groups
\begin{itemize}
\item[(i)] of order $p^{6k}$ for all $p$ and $k$;
\item[(ii)] of order $p^{10}$ if and only if $p^5\equiv 1\pmod{11}$;
\item[(iii)] of order $p^{14k}$ for all $p$ and $k$.
\end{itemize} 
\end{enumerate}
\end{theorem}

Parts~(a) and~(b) of Theorem~\ref{main} follow from Theorem~\ref{r23}, while part~(c)
follows from Theorems~\ref{4th} and~\ref{4gen}.
Part~(d) is verified in Theorem~\ref{thm-expsquare} 
and the proof of~(e)
can be found at the end of Section~\ref{s7}.

As UCS $p$-groups have precisely three characteristic subgroups, 
they must have exponent $p$ or $p^2$.  
Abelian UCS $p$-groups are all of the form
$(C_{p^2})^r$ where $p$ is a prime.
Non-abelian UCS $p$-groups with exponent $p$ are quite common, and so
we could investigate them only for small generator number 
(at most 4). However, it seems that non-abelian UCS $p$-groups with 
exponent $p^2$ are less common, and we could even prove that for certain
choices of the pair $(p,r)$ in Theorem~\ref{main} 
they do not exist. For instance Theorem~\ref{main}(d) implies that there is
no non-abelian UCS 5-group with 4~generators and exponent 25.
In both cases, the study of such groups leads to challenging
problems in representation theory. A well-written introduction to
$p$-groups of Frattini length~2 can be found in \cite{GQ06}.

Let $p$ be an odd prime, and let $G$ be a non-abelian $r$-generator 
UCS $p$-group with exponent $p$. Then $G$ has the form
$H/N$ where $H$ is the $r$-generator, free group with exponent $p$ and 
nilpotency class~2, and $N$ is a proper 
subgroup of $H'$. The groups $H/H'$ and $H'$
are elementary abelian, and so they can be considered as vector spaces 
over $\F_p$. Moreover, the group $\GL_r(p)$ acts on $H/H'$ in the natural 
action, and on $H'$ in the exterior square action.
It is proved in Theorem~\ref{3th} that $G$ is UCS if and only if the 
(setwise) stabilizer $K:=\GL_r(p)_N$ is irreducible on both $H/H'$ and $H'/N$. Further,
an irreducible subgroup $K\leq\GL_r(p)$ and an irreducible $K$-factor module 
$H'/N$
lead to a UCS $p$-group with exponent $p$ (Theorem~\ref{UCSconstr}). 
Thus the investigation of the class of UCS 
$p$-groups with exponent $p$ is reduced to the study of irreducible
subgroups $K$ of $\GL_r(p)$ and the maximal $K$-submodules
of the exterior square~$\Lambda^2(\F_p)^r$.

The study of UCS 2-groups poses different problems than the odd case. In this
paper we concentrate on UCS groups of odd order, and prove only a few results
about UCS 2-groups.

A non-abelian UCS $p$-group $G$ with \emph{odd} order and exponent $p^2$,
is a {\it powerful} $p$-group (i.e. $G^p\leq G'$). Moreover, $G$ has
the form $H/N$ where $H$ is the $r$-generator free group with
$p$-class~2 and exponent~$p^2$ and $N$ is a subgroup of $\Phi(H)$. 
As above, $\GL_r(p)$ acts on 
$H/\Phi(H)$ in the natural action and $\Phi(H)=H^p\oplus H'$
can also be viewed as a $\GL_r(p)$-module.
The commutator subgroup $G'$ is isomorphic to $H'/(H'\cap N)$ and,
as $G$ has one proper non-trivial characteristic subgroup, 
$G'$ must coincide with $G^p$. On the other hand, as the $p$-th power 
map $x\mapsto x^p$ is a homomorphism, the $\GL_r(p)_N$
actions on $H/\Phi(H)$ and on $H'/(H'\cap N)$ must be equivalent.
Further
$\GL_r(p)_N$ must be irreducible on both
$H/\Phi(H)$ and $H'/(H'\cap N)$ 
(see Theorem~\ref{3th} for the proof of the last two assertions).  
Therefore we found an irreducible subgroup
$K$ of $\GL_r(p)$ and a maximal $K$-submodule $N$ in $\Lambda^2(\F_p)^r$
such that $\Lambda^2(\F_p)^r/N$ is equivalent to the natural module of 
$K$. Conversely, such a group $K$ and a submodule $N$ lead to a UCS 
$p$-group with exponent $p^2$ (see Theorem~\ref{UCSconstr}).
Thus UCS $p$-groups with odd exponent $p^2$ give rise to irreducible modules
that are isomorphic to a quotient of the exterior square. We call these
{\it exterior self-quotient} modules (or ESQ-modules). We are planning
to write a paper devoted to the study of ESQ-modules.

In Sections~\ref{s3.5} and~\ref{s5}, UCS $p$-groups with exponent $p$
and generator number at most~4 are studied.
We achieve a complete classification of 2- and 3-generator UCS $p$-groups and,
for odd $p$, a complete classification of $4$-generator UCS $p$-groups with
exponent $p$.
These classifications are made possible in these cases by our knowledge of 
the $\GL_r(p)$-module $W=\Lambda^2(\F_p)^r$. If $r=2$ then $W$ is a 
1-dimensional module and our problem is trivial. 
If $r=3$, then $W$ is equivalent to the dual of the natural
module and the classification of UCS $p$-groups with exponent~$p$ is 
straightforward also in this case. However, the classification of 
$3$-generator UCS $p$-groups with exponent $p^2$ is already 
non-trivial. The fact that there is, up to isomorphism, precisely one
3-generator UCS $p$-group with exponent $p^2$ is a consequence of the fact that
the stabilizer in the general linear group $\GL_3(p)$ of a certain
3-dimensional subspace of $(\F_p)^6$ is the special orthogonal
group $\textup{SO}_3(p)$ (see Lemma~\ref{solemma}).

When $r=4$, the Klein correspondence makes it possible to
obtain the necessary information about the $\GL_r(p)$-module $W$. In this
case $\GL_r(p)$ preserves (up to scalar multiples) a quadratic form on $W$. 
A $p$-group with exponent $p$ corresponds to a subspace $N$ of 
$W$ as explained above. 
These observations and the classification of $p$-groups
with order dividing $p^7$ (see~\cite{p6,p7})
enables us to classify $4$-generator UCS $p$-groups with exponent $p$.
We found it surprising that $N$ leads to a 
UCS $p$-group if and only if the restriction of
the quadratic form to $N$ is non-degenerate.
The details are presented  in Section~\ref{s5}. A brief 
classification of 4-generator UCS 2-groups is given in Section~\ref{s4.5}.

Section~\ref{s7} focusses on the construction of exterior self-quotient (ESQ)
modules for dimensions at most 5. As remarked above, this is directly
related to the construction of UCS $p$-groups with odd exponent $p^2$.
Our results for exponent-$p^2$ UCS groups are summarized in 
Theorem~\ref{main}(d,e), and proved in Section~\ref{s7}.
Theorems~\ref{d=5},~\ref{d=4} are concerned with the structure of
ESQ-modules in dimensions 4 and 5, and are of independent interest.


\section{$p$-groups with precisely 3 characteristic subgroups}\label{s3}

We shall focus henceforth on the structure of a $p$-group $G$ with
precisely three characteristic subgroups. In such a group the Frattini
subgroup $\Phi(G)$ is non-trivial, otherwise $G$ is elementary abelian
and characteristically simple. Therefore the non-trivial,
proper characteristic subgroup of $G$ is $\Phi(G)$. 

Let $p$ be a prime number. 
The {\em lower $p$-central series} of a group $H$ (which may
not be a $p$-group) is
defined as follows: $\lambda_1^p(H)=H$, and
$\lambda^p_i(H)=[\lambda^p_{i-1}(H),H](\lambda^p_{i-1}(H))^p$ for $i\geq 2$. 
For a finite $p$-group $G$,
the subgroup $\lambda_i^p(G)$ is denoted by $\lambda_i(G)$. 
The term 
$\lambda_2(G)$ coincides with $\Phi(G)$. The {\em
  $p$-class} of a finite $p$-group $G$ 
is defined as the smallest integer $c$ such that
$\lambda_{c+1}(G)=1$. 
Clearly, terms of the lower $p$-central series of a finite $p$-group $G$ 
are characteristic, and 
$G$ has $p$-class~1 if and only if $G$ is elementary abelian.  This simple
observation leads to the following lemma.

\begin{lemma}\label{pc2}
The $p$-class of a UCS $p$-group is precisely~$2$.
\end{lemma}

A UCS $p$-group can be written as a quotient of a suitable free group which we
now describe.
Let $p$ be a prime, and let $r$ be a positive integer. 
Let $F_r$ denote the free group with rank $r$, and set
$$
H_{p,r}=F_r/\lambda^p_3(F_r).
$$ 
The group $H_{p,r}$ is an $r$-generator free group in the variety of groups
with $p$-class~2. 
Since $\lambda_3^p(H_{p,r})=1$, we have that $H_{p,r}$ is an $r$-generator
nilpotent group
whose exponent divides $p^2$, and so $H_{p,r}$ is a finite $p$-group.
The quotient 
$H_{p,r}/\Phi(H_{p,r})$ is elementary abelian with rank $r$.
As $\lambda_3(H_{p,r})=[\Phi(H_{p,r}),H_{p,r}]\Phi(H_{p,r})^p=1$, the subgroup $\Phi(H_{p,r})$ is elementary 
abelian and central. 
Assume that the elements 
$x_1,\ldots,x_r$ form a minimal generating set for $H_{p,r}$. The Frattini subgroup 
$\Phi(H_{p,r})$ is minimally generated by the elements $x_i^p$ and 
$[x_j,x_k]$ with $i,\ j,\ k\in\{1,\ldots,r\}$ and $j<k$. Thus $\Phi(H_{p,r})$ 
has rank $r(r-1)/2+r$.
In $2$-groups, 
the subgroup generated by the squares contains the commutator subgroup, 
so $\Phi(H_{2,r})=(H_{2,r})^2$. 
On the other hand, if $p\geq 3$, then $\Phi(H_{p,r})=(H_{p,r})'\oplus (H_{p,r})^p$. 
In this case,
the commutators $[x_j,x_k]$ with $j<k$ form a minimal generating set for 
$(H_{p,r})'$,  while the $p$-th powers $x_i^p$ form a minimal generating set for $(H_{p,r})^p$. 
Thus $(H_{p,r})'$ and $(H_{p,r})^p$ are elementary abelian with ranks $r(r-1)/2$ and~$r$, respectively.

When investigating UCS $p$-groups $G$, by the following lemma,
we may conveniently assume
that $G$ is of the form $H_{p,r}/N$ 
where $N\leq \Phi(H_{p,r})$. The proof of the following lemma
is straightforward.

\begin{lemma}\label{idgh}
Let $p$ be a prime, let $G$ be an $r$-generator 
UCS $p$-group, and let 
$H_{p,r}$ be the group above. 
Suppose that
$\{x_1,\ldots,x_r\}$ and 
$\{y_1,\ldots,y_r\}$ are minimal generating sets of $H_{p,r}$ and $G$, 
respectively. Then the mapping
$x_i\mapsto y_i$, for $i\in\{1,\ldots,r\}$, can uniquely be extended to an 
epimorphism $\varphi:H_{p,r}\rightarrow G$. Further, the kernel 
of $\varphi$ lies in $\Phi(H_{p,r})$. 
\end{lemma}

Now we introduce some notation that will be used in the rest of the paper.
If $L$~is a group that 
acts on a vector space $V$, then $L^V$ denotes the image of $L$ under 
this action. Hence $L^V\leq\GL(V)$.
The stabilizer in $L$ of an object $X$ is denoted by $L_X$. 
If $G$ is a $p$-group, then $\overline G$ denotes $G/\Phi(G)$.
If $G$ is a UCS $p$-group, then
$\Phi(G)$ and $\overline G$ can be considered as $\F_p$-vector spaces. 
We shall consider the linear groups $\Aut(G)^{\Phi(G)}$  and 
$\Aut(G)^{\overline G}$. 

Set $H=H_{p,r}$. The subgroup $\Phi(H)$ can be 
viewed as a $\GL(\overline{H})$-module as follows. 
Recall that
$H=F_r/\lambda_3^p(F_r)$ and $\Phi(H)=\lambda_2^p(F_r)/\lambda_3^p(F_r)$.
Let $g\in \GL(F_r/\lambda_2^p(F_r))$ 
and $x,\ y\in F_r$. 
Set $(x^p\lambda_3^p(F_r))g=\hat x^p\lambda_3^p(F_r)$ and 
$([x,y]\lambda_3^p(F_r))g=[\hat x,\hat y]\lambda_3^p(F_r)$
where $\hat x$ and $\hat y$ are chosen so that $(x\lambda_2^p(F_r))g=\hat x\lambda_2^p(F_r)$ and $(y\lambda_2^p(F_r))g=\hat y\lambda_2^p(F_r)$. 
It is proved in~\cite[Lemma~2.6]{obrien} 
that this rule defines a unique linear transformation
on $\Phi(H)$ corresponding to $g$.

The following lemma describes
the structure of the $\GL(\overline H)$-module $\Phi(H)$; see for
instance the argument on page~26 in~\cite{hig}. 

\begin{lemma}\label{strmodules}
Let $p$ be a prime, let $r$ be an integer, and set $H=H_{p,r}$.
The subgroup 
$H'$ is a $\GL(\overline H)$-submodule of $\Phi(H)$ and 
$\Phi(H)/H'\cong \overline H$ as $\GL(\overline H)$-modules. If 
$p\geq 3$, then $\Phi(H)=H'\oplus H^p$ is a
direct sum of $\GL(\overline H)$-modules.
In particular, $\overline H\cong H^p$ as $\GL(\overline H)$-modules if $p\geq3$.
\end{lemma}

Next we give a description of the automorphism group of $G$
following~\cite[Theorem~2.10]{obrien}. If $G$ is of the form $H_{p,r}/N$,
with some $N\leq \Phi(H_{p,r})$ then the spaces $\overline{H_{p,r}}$ and 
$\overline G$ 
can naturally be identified, and this fact is exploited in the following lemma.

\begin{lemma}\label{autgrp}
Let $p$ be a prime, let $r$ be an integer, and set $H=H_{p,r}$. 
Let $N$ be a subgroup
of $\Phi(H)$, and set $G=H/N$. 
Identifying $\overline G$ and $\overline H$, we obtain that
$$
\Aut(G)^{\overline G}=\GL(\overline H)_N\quad\mbox{and}\quad
\Aut(G)^{\Phi(G)}=(\GL(\overline H)_N)^{\Phi(H)/N}.
$$ 
Further, the kernel $K$ of the $\Aut(G)$-action on $\overline G$ is an
elementary abelian $p$-group of order $|\Phi(G)|^r$, and $K$ acts
trivially on $\Phi(G)$. 
\end{lemma}

Lemma~\ref{autgrp} enables us to characterize UCS $p$-groups.

\begin{theorem}\label{3th}
Let $p$ be a prime, let $r$ be an integer, set $H=H_{p,r}$, and
let  $G=H/N$ where $N\leq\Phi(H)$. 
Then the following are equivalent:
\begin{itemize}
\item[(a)] $G$ is a UCS $p$-group; 
\item[(b)] both $\Aut(G)^{\overline G}$ 
and $\Aut(G)^{\Phi(G)}$ are
irreducible;
\item[(c)] both $\GL(\overline H)_N$ and 
$(\GL(\overline H)_N)^{\Phi(H)/N}$ are irreducible.
\end{itemize}
Further, if $p$ is odd and 
$G$ is a UCS $p$-group, then precisely one of the following must
hold:
\begin{enumerate}
\item[(i)] $N=H'$ and $G$ is abelian;
\item[(ii)] $H^p\leq N$ and $G$ is non-abelian of exponent $p$;
\item[(iii)] $N\cap H^p=1$, $\overline H$ 
and $H'/(N\cap H')$ are equivalent 
$\GL(\overline H)_N$-modules, $G$ is non-abelian of exponent $p^2$, 
and $|G|=p^{2r}$. 
\end{enumerate}
\end{theorem}
\begin{proof}
Assertions~(b) and~(c) are equivalent by Lemma~\ref{autgrp}. We now prove
that~(a) and~(b) are equivalent. The following two observations
show that~(a) implies~(b). First, inverse images of
$\Aut(G)^{\overline G}$-submodules of $\overline G$ correspond bijectively to
characteristic subgroups of $G$ containing $\Phi(G)$. Second,
$\Aut(G)^{\Phi(G)}$-submodules of $\Phi(G)$ correspond bijectively to
characteristic subgroups of $G$ contained in $\Phi(G)$. Assume now
that~(b) holds and that $L$ is a characteristic subgroup of $G$. As $L\Phi(G)$
is characteristic, it follows from the observation above that $L\Phi(G)$ equals $\Phi(G)$ or $G$.
In the latter case, $L=G$ as $\Phi(G)$ comprises the set of elements of $G$
that can be omitted from generating sets. In the former case $L\Phi(G)=\Phi(G)$
and $L\leq\Phi(G)$. It follows from~(b) that $L$ equals $1$ or $\Phi(G)$.
Thus~(b) implies~(a).

We now prove the second statement of the theorem. Suppose that $G$ is a
UCS $p$-group, where $p$ is odd. Suppose additionally that $G$ is abelian. Then
$H'\leq N$. Now $H/H'\cong (C_{p^2})^r$ is homocyclic of rank~$r$
and exponent~$p^2$. As in a UCS $p$-group  the subgroups $G^p$ and $\{g\in
G\mid g^p=1\}$ coincide, we obtain that $N=H'$ and~(i) holds.
Suppose now that $G$ is non-abelian.
By Lemma~\ref{strmodules}, $H'N$ is invariant 
under $\GL(\overline H)_N$. By Lemma~\ref{autgrp}, 
$\Aut(G)^{\Phi(G)}=(\GL(\overline H)_N)^{\Phi(H)/N}$, which shows that 
$H'N/N$ is  characteristic in $G$.
Thus $H'N$ equals $\Phi(H)$ or $N$. As $G$ is non-abelian $H'N=\Phi(H)$ holds.
As $p$ is odd, Lemma~\ref{strmodules} implies that $H^p$ is 
invariant under $\GL(\overline H)$, and so $H^p\cap N$ must be 
invariant under $\GL(\overline H)_N$. On the other hand, the first part of 
Theorem~\ref{3th} shows 
that $\GL(\overline H)_N$ is irreducible on $\overline H$, 
and so, by Lemma~\ref{strmodules}, also on $H^p$. 
Thus either $N\cap H^p=H^p$ 
or $N\cap H^p=1$. In the
former case $H^p\leq N$, and case~(ii) holds. In the
latter case, we shall show below that case~(iii) holds.

Suppose now that $H'N=\Phi(H)$, $N\cap H^p=1$ and $p$ is odd. 
As $H^p$ is invariant under $\GL(\overline H)$ (see Lemma~\ref{strmodules}), 
$NH^p/N$ must be invariant under $\GL(\overline H)_N$. Thus Lemma~\ref{autgrp}
implies that $NH^p/N$ is characteristic in $G$.
Hence $NH^p=\Phi(H)$. 
As $N\cap H^p=1$, we obtain the following isomorphisms between $\GL(\overline H)_N$-modules:
$$
H^p\cong\frac{H^p}{N\cap H^p}\cong
\frac{H^pN}{N}=\frac{\Phi(H)}{N}=\frac{H'N}{N}\cong\frac{H'}{N\cap H'}.
$$
By Lemma~\ref{strmodules}, $H^p\cong \overline H$, as 
$\GL(\overline H)$-modules, so $\overline H$ and
$H'/(N\cap H')$ are equivalent $\GL(\overline H)_N$-modules.
It follows from the above displayed equation that 
$|N|=|H'|=p^{r(r-1)/2}$, and hence that $|G|=p^{2r}$. Thus case~(iii) holds.
\end{proof}

\section{The existence of UCS $p$-groups}\label{ucsexist}

In this section we study the question whether UCS $p$-groups exist with
given parameters. We find that the existence of UCS $p$-groups is equivalent
to the existence of certain irreducible linear groups.

Let $H=H_{p,r}$ for some prime $p$ and integer $r$ as above. 
Note that the $\GL(\overline{H})$-action on $H'$ is equivalent to the 
exterior square of the natural action. If $p$ is an odd prime 
and $G$ is an $r$-generator  
UCS $p$-group with exponent $p$, then $G$ is non-abelian and has the form 
$H/N$ where $H^p\leq N< \Phi(H)$. 
Theorem~\ref{3th} implies that $\GL(\overline H)_N$ must be irreducible on 
$\overline H$ and also on $\Phi(H)/N$. As $H'\not\leq N$, we obtain that
$$
\frac{\Phi(H)}{N}=\frac{H'N}{N}\cong
\frac{H'}{H'\cap N}.
$$
Denote $\overline G$ by $V$. The argument 
above shows that an exponent-$p$ UCS $p$-group $G$
gives rise to an irreducible linear group $K:=\Aut(G)^{\overline G}$
acting on $V$, and a maximal $K$-module $M$ of $\Lambda^2 V$.

The structure of UCS $p$-groups with exponent $p^2$ is, by
Theorem~\ref{3th}(iii), intimately related to the following (apparently new)
concept in representation theory.

\noindent {\bf Definition.} 
An $\F G$-module $V$ is called an {\em exterior self-quotient} module,
briefly an {\em ESQ}-module, if there is an $\F G$-submodule $U$ of
$\VwedgeV$ such that $(\VwedgeV)/U$ is isomorphic to~$V$. If
$G$ acts faithfully on $V$, we call
$G$ an {\em ESQ-subgroup} of $\GL(V)$, or simply an {\em ESQ-group}.

By Theorem~\ref{3th}, $\Aut(G)^{\overline G}$ is
an irreducible ESQ-group when $G$ is a non-abelian 
UCS $p$-group of exponent $p^2$ for odd $p$. In Section~\ref{s7} we
study exterior self-quotient modules, where it is natural to
consider fields other than $\F_p$. When we consider {\em necessary} conditions
for the existence of UCS $p$-groups then we usually work over 
a finite prime field $\F_p$; however,
for {\em sufficient} conditions, working over 
arbitrary (finite) fields is most natural.

Suppose that $V$ is a $d$-dimensional 
vector space over a field $\F_q$ where $q=p^k$ for some prime $p$ and 
integer $k$. We may consider $V$ as a vector space over the prime field $\F_p$
and also over the larger field $\F_{q}$. Thus we may take the exterior
squares $\Lambda^2_{\F_p}V$ and $\Lambda^2_{\F_{q}}V$ over $\F_p$ and $\F_{q}$,
respectively. There is an $\F_p$-linear map 
$\varepsilon:\Lambda^2_{\F_p} V
\rightarrow \Lambda^2_{\F_q} V$ satisfying 
$\varepsilon(u\w v)= u\w v$ for all
$u,v\in V$. 

\begin{theorem}\label{UCSconstr}
Let $p$ be an odd prime, let $r$ and $s$ be integers.

\textup{(a)} 
The following two assertions are equivalent.
\begin{enumerate}
\item[\textup{(a1)}] There exists a UCS $p$-group $G$ with exponent $p$ such that 
$|\overline G|=p^{r}$ and $|\Phi(G)|=p^s$. 
\item[\textup{(a2)}] There exists an 
irreducible linear group $K$ acting on a vector space $V$ over $\F_{p^k}$, 
for some $k$, such that $\dim V=r/k$ and 
$\Lambda^2_{\F_{p^k}} V$ has a maximal $\F_{p^k}K$-submodule with codimension
$s/k$.
\end{enumerate}

\textup{(b)} The following two assertions are equivalent.
\begin{enumerate}
\item[\textup{(b1)}] There exists a UCS $p$-group $G$ with exponent $p^2$ such that 
$|\overline G|=p^{r}$.
\item[\textup{(b2)}] 
There exists an irreducible ESQ-module $V$ over a field $\F_{p^k}$ such
that $\dim V=r/k$, and $V$ can not be written over any proper subfields of
$\F_{p^k}$.
\end{enumerate}
\end{theorem}
\begin{proof}
(a) 
Assume (a1) is true, and $G$ is an exponent-$p$ UCS-group. Then by
Theorem~\ref{3th} 
$\Aut(G)^{\overline G}$ is an irreducible linear group, so (a2) is true
with $k=1$. Assume now that (a2) holds and set $q=p^k$. Let 
$U$ be a maximal $\F_{q}K$-submodule of $\Lambda^2_{\F_q} V$ of codimension $s/k$.
Let $\varepsilon$ denote the epimorphism
$\Lambda^2_{\F_p} V\rightarrow \Lambda^2_{\F_q} V$. Let $Z$ denote the group
of the non-zero scalar transformations
$\{\lambda I\mid \lambda\in\F_{q}^\times\}$
of $V$. Then $Z$ commutes with $K$ and 
so one can form the subgroup $ZK$. 
Since $K$ is irreducible on $V$ over $\F_q$ and the action of 
$\F_{q}^\times$ on $V$ is realized by $Z$, we find that $ZK$ is 
irreducible on $V$ over $\F_p$. We claim that
$ZK$ is irreducible on $\Lambda^2_{\F_q} V/U$ over $\F_p$. 
Note that the element $\lambda I\in Z$ induces the
scalar transformation $\lambda^2 I$ on $\Lambda^2_{\F_q}
V/U$. Since these transformations generate the $\F_p$-algebra $\{\lambda I\ |\
\lambda\in\F_q\}$ of all scalar
transformations of $\Lambda^2_{\F_q} V/U$, we obtain that an $\F_pZK$-submodule
of $\Lambda^2_{\F_q} V/U$ is also an $\F_qZK$-submodule. Since $K$ is
irreducible on $\Lambda^2_{\F_q} V/U$ over $\F_q$, we obtain that $ZK$ is 
irreducible on $\Lambda^2_{\F_q} V/U$ over $\F_p$.

Obviously, $\varepsilon$ is a $ZK$-homomorphism.
If $\hat{U}=\varepsilon^{-1}(U)$, then clearly 
$\Lambda^2_{\F_q} V/U\cong
\Lambda^2_{\F_p} V/{\hat U}$.
Hence $\hat U$ is a maximal $ZK$-submodule in $\Lambda^2_{\F_p} V$. 
By this argument, we may, and shall, henceforth assume
that (a2) is true with $k=1$ and will write $\Lambda^2$ for $\Lambda^2_{\F_p}$.

Let $H$ denote $H_{p,r}$. As the $\GL(\overline H)$-action on 
$H'$ is equivalent to its action on $\Lambda^2 \overline H$, we
identify $V$ with $\overline H$, $\Lambda^2 V$ with $H'$, and $\hat U$
with a $K$-invariant normal subgroup $N$ of index $p^s$ in $H'$.
Set $G=H/(H^pN)$.
As $p\geq 3$, Lemma~\ref{strmodules} shows that $H^p$ and
$H'$ are $\GL(\overline H)$-submodules
of $\Phi(H)$ such that $\Phi(H)=H'\oplus H^p$. 
Hence $K$ must stabilize $H^pN$ and so Lemma~\ref{autgrp} gives that
$K\leq \Aut(G)^{\overline G}$. 
Since $K$ is irreducible, so is $\Aut(G)^{\overline G}$. 
As
$$
\frac{\Phi(H)}{H^pN}=\frac{H^pH'}{H^pN}
\cong\frac{H'}{N},
$$
we obtain that
$K$ and $\Aut(G)^{\overline G}$ are irreducible on
$\Phi(H)/(H^pN)$.
Now Theorem~\ref{3th} implies that $G$ is a UCS $p$-group with exponent $p$.

(b) The discussion at the beginning of this section shows that
(b1) implies (b2) with $k=1$. Before proving the converse, we argue that
we may assume that $k=1$. Suppose that $V$ is an ESQ $\F_q K$-module, and
$U$ is an $\F_q K$-submodule satisfying $\Lambda^2 V/U\cong V$. Using the
notation in part~(a), the $\F_p(ZK)$-homomorphism
$\varepsilon\colon \Lambda^2_{\F_p} V\to\Lambda^2_{\F_q} V$ gives rise to
an $\F_p(ZK)$-isomorphism
$\Lambda^2_{\F_p} V/{\hat U}\cong \Lambda^2_{\F_q} V/U$ where
$\hat{U}:=\varepsilon^{-1}(U)$. Since $\Lambda^2 V/U\cong V$ is an
$\F_pK$-isomorphism, it follows that $\Lambda^2_{\F_p} V$ is ESQ. Moreover,
$V$ is an irreducible $\F_pK$-module by \cite[Theorem~VII.1.16(e)]{HB}.
In summary, viewing $V$ as a $K$-module over $\F_p$ of larger dimension
allows us to assume that the hypotheses for (b2)
hold for $k=1$.

Suppose that $k=1$. Set $H=H_{p,r}$. Take $K$ to be an irreducible 
subgroup of $\GL(\overline H)$, and $M$ to be a $K$-submodule of $H'$ such
that $\overline H$ and $H'/M$ are isomorphic.
Specifically let $\varphi\colon{\overline H}\to H'/M$ be a $K$-module isomorphism. Set
$$
N=\left\{x^py\ |\ x\in H,\ y\in H'\mbox{ such that }
 \varphi(x\Phi(H))=yM\right\},
$$
and set $G=H/N$. As $p\geq 3$, one can easily check that $N$ is a subgroup of 
$\Phi(H)$ and that $H'\cap N=M$.  
Lemma~\ref{strmodules} shows that the  map 
$x\Phi(H)\mapsto x^p$ is an isomorphism between the 
$\GL(\overline H)$-modules $\overline H$ and $H^p$. Therefore,
if $g\in K$ and $x^py\in N$ with some $x\in H$ and $y\in H'$ 
then, as $\varphi$ is a $K$-homomorphism, $(x^py)^g=(x^g)^py^g\in N$.
Thus $N$ is a $K$-submodule. Therefore Lemma~\ref{autgrp} shows
that $K\leq\Aut(G)^{\overline G}$. By assumption $K$ is irreducible on 
$\overline H\cong\overline G$. The definition of $N$ gives that 
$H'N=\Phi(H)$, and so
$$
\frac{\Phi(H)}N=\frac{H'N}{N}\cong
\frac{H'}{H'\cap N}=\frac{H'}{M}\cong \overline G.
$$
As $\Phi(G)\cong\Phi(H)/N$, we obtain that $K$ acts  
irreducibly on
$\Phi(G)$. Since  $K\leq \GL(\overline H)_N$, this shows that 
$(\GL(\overline H)_N)^{\Phi(H)/N}$ is irreducible. Hence, 
by Theorem~\ref{3th}, $G$ must be a UCS $p$-group. Since
$N\cap H^p=1$, $G$ has exponent~$p^2$. Thus (b2) implies (b1).
\end{proof}

\section{UCS $p$-groups with generator number at most 3}\label{s3.5}

In this section we classify 2- and 3-generator UCS $p$-groups.
The main result of this section is the following theorem from which
Theorem~\ref{main}(a)-(b) follows.

\begin{theorem}\label{r23}
Let $G$ be an $r$-generated non-abelian UCS $p$-group.
\begin{itemize}
\item[(a)] 
If $p=r=2$, then $G$ is isomorphic to the quaternion group $Q_8$,
and if $p=2$ and $r=3$, then $G\cong G_1$ where
$$
G_1=\left<x_1,\ x_2,\ x_3\ |\
x_1^2[x_1,x_2][x_1,x_3][x_2,x_3],\ x_2^2[x_1,x_2][x_1,x_3],\ x_3^2[x_1,x_2],
\ \mbox{$2$-class $2$}
\right>
$$
has order $2^6$. Further, $\Aut(Q_8)^{\overline{Q_8}}\cong\GL_2(2)$ and
$\Aut(G_1)^{\overline{G_1}}$ has order $21$. 

\item[(b)] 
If $p\geq 3$ and $r=2$, then $G\cong G_2$ where
$
G_2=\langle x_1,\ x_2\ |\  x_1^p,\ x_2^p,\ \mbox{$p$-class $2$}\rangle
$
is extraspecial of order $p^3$ and exponent $p$.
Further, $\Aut(G_2)^{\overline{G_2}}\cong\GL_2(p)$. 
\item[(c)] If $p\geq 3$ and $r=3$, then $G$ has order $p^6$ and $G\cong G_3$ or $G_4$ where
\begin{align*}
  G_3&=\langle x_1,\ x_2,\ x_3\ |\  x_1^p,\ x_2^p,\ x_3^p,\ \mbox{$p$-class $2$}\rangle, \quad\mbox{and}\\
  G_4&=\langle x_1,\ x_2,\ x_3\ |\
x_1^p=[x_2,x_3],\ x_2^p=[x_3,x_1],\ x_3^p=[x_1,x_2],\ \mbox{$p$-class $2$}\rangle.
\end{align*}
Further, $\Aut(G_3)^{\overline{G_3}}\cong\GL_3(p)$ and $\Aut(G_4)^{\overline{G_4}}\cong\SO_3(p)$. 
\end{itemize}
\end{theorem}

\begin{proof}
(a) If $p = 2$ and $r = 2$, then
$|G'| = 2$ and so $|G| = 8$. The dihedral group of order~$8$ has a
characteristic cyclic subgroup of order $4$, so the only possibility is
that $G\cong Q_8$. The group $Q_8$ can be written as $H_{2,2}/N$ where 
$N=\left<x^2[y,x],y^2[y,x]\right>$ and $x,\ y$ are the generators of
$H_{2,2}$. Now easy calculation shows that $N$ is invariant under
$\GL(\overline{H_{2,2}})$,
and so Lemma~\ref{autgrp} implies that $\Aut{(Q_8)}^{\overline{Q_8}}\cong\GL_2(2)$.
Let now $p = 2$, $r = 3$. It can be
checked that every irreducible subgroup of $\GL_3(2)$ has order divisible
by~7. All subgroups of order~7 are conjugate in $\GL_3(2)$, so we may take
an arbitrary one. Its action on $\Phi(H_{2,3})$ is the sum of two non-isomorphic
irreducible 3-dimensional submodules, say $(H_{2,3})'$ and $N$. Hence
$H_{2,3}/N$ is a non-abelian UCS-group. A direct calculation (or an
application of a computational algebra system \cite{GAP,Magma})
shows that $H_{2,3}/N\cong G_1$, and $\Aut(G_1)^{\overline{G_1}}$ is
a non-abelian subgroup of $\GL_3(2)$ of order~21.

(b) 
Suppose that $p$ is odd, and $H=H_{p,2}$. Let $N$ be a subgroup of
$\Phi(H)$ such that $G=H/N$ is a non-abelian UCS $p$-group. As $G'=\Phi(G)$
has order $p$, it follows that $N=H^p$. Moreover, $H/H^p$ is an
exponent-$p$ UCS extraspecial group isomorphic to $G_2$.
By Lemma~\ref{strmodules}, $H^p$ 
is invariant under $\GL(\overline{H})$, and so
$\Aut(G)^{\overline G}\cong\GL_2(p)$.

This completes the proof of parts~(a) and~(b). The proof of part~(c) relies
on the following lemma.

\begin{lemma}\label{solemma}
Suppose that $V$ is a $3$-dimensional vector space over a finite field $\F$,
where $\textup{char}(\F)\ne2$.
Let $U$ be a subspace of $V\oplus \Lambda^2 V$ such that $\dim U=3$, 
$U\cap V=U\cap \Lambda^2 V=0$ and that $\GL(V)_U$ is irreducible on $V$. Then
there exists a $g\in\GL(V)$ such that $Ug=W$ where
$$
W=\left<e_1-e_2\w e_3,e_2-e_3\w e_1,e_3-e_1\w e_2\right>. 
$$
and $e_1,e_2,e_3$ is a basis for $V$.
Further, $\GL(V)_U=g\GL(V)_W g^{-1}$ and $\GL(V)_W=\SO_3(\F)$. 
\end{lemma}

\begin{proof}
Let $e_1,e_2,e_3$ be a basis for $V$, and let
$e_2\wedge e_3,e_3\wedge e_1,e_1\wedge e_2$ be the corresponding
dual basis for $\VwedgeV$. Concatenating these bases gives a basis
for $V\oplus \VwedgeV$. We view $g\in\GL(V)$ as a $3\times3$ matrix
relative to the basis $e_1,e_2,e_3$. An easy computation shows that 
the transformation $g\wedge g\in\GL(\VwedgeV)$ defined by
$(u\wedge v)(g\wedge g)=(ug)\wedge(vg)$, has matrix
$\det(g)(g^{-1})^\T=\det(g)g^{-\T}$ relative to the above dual basis. Thus
$g$ acting on $V\oplus \VwedgeV$ has matrix
\begin{equation}\label{Mat}
  \begin{pmatrix}g&0\\0&g\wedge g\end{pmatrix}
  =\begin{pmatrix}g&0\\0&\det(g)g^{-\T}\end{pmatrix}.
\end{equation}

The subspace $U$ has a basis of the form $e_1-a_1,e_2-a_2,e_3-a_3$ where
$a_1,a_2,a_3$ is a basis for $\VwedgeV$. We now calculate the
stabilizer $\GL(V)_U$. 
Let $U_A$ denote the $3\times 6$ matrix 
$\begin{pmatrix}I\mid -A\end{pmatrix}$
where $A$ is the
invertible $3\times3$ matrix with $i$th row
\[
  (a_{i1},a_{i2},a_{i3})\quad\text{where}\quad
    a_i=a_{i1}e_2\wedge e_3+a_{i2}e_3\wedge e_1+a_{i3}e_1\wedge e_2.
\]
Note that the matrix $\begin{pmatrix}I\mid -A\end{pmatrix}$ possesses
two $3\times 3$ sub-blocks.
We shall view $U$ as the row space of $U_A$.
Let $g$ be an arbitrary invertible $3\times3$ matrix. Then
\[
\begin{pmatrix}I\mid -A\end{pmatrix}
\begin{pmatrix}g&0\\0&g\wedge g\end{pmatrix}
  =
\begin{pmatrix}g\mid -A(g\wedge g)\end{pmatrix}
\]
However, the matrices $\begin{pmatrix}g\mid -A(g\wedge g)\end{pmatrix}$
and
$\begin{pmatrix}I\mid -g^{-1}A(g\wedge g)\end{pmatrix}$
have the same row space.
The image of $U_A$ under the $\GL(V)$-action is thus
$(U_A)g=U_{g^{-1}A(g\wedge g)}=U_{\det(g)g^{-1}Ag^{-\T}}$,
by equation~(\ref{Mat}). Hence $g\in\GL(V)$ stabilizes $U=U_A$
if and only if $A=g^{-1}A(g\wedge g)$, or $A$ intertwines $g$ and $g\wedge g$.
In summary, $g\in\GL(V)_U$ if and only if $gAg^\T=(\det g)A$.

The stabilizer $\GL(V)_U$ is contained in the subgroup
\[
  \Gamma=\{g\in\GL(V)\mid g (A-A^\T) g^\T=(\det g)(A-A^\T)\}.
\]
However $\Gamma$, and hence $\GL(V)_U$, fixes the null space
\[
  \{v\in V\mid v(A-A^\T)=0\}.
\]
As $\GL(V)_U$ acts irreducibly on $V$, $A-A^\T$ is either $0$ or
invertible. Since
\[
  \det(A-A^\T)=\det((A-A^\text{T})^\T)=\det(A^\T-A)=(-1)^3\det(A-A^\T)
\]
and $\text{char}(\F)\ne2$, we see that $\det(A-A^\T)=0$.
Thus $A-A^\T=0$, and $A$ is
symmetric. Given that $A$ is invertible, $gAg^\T=(\det g)A$ for
$g\in\GL(V)_U$ implies $\det(g)=1$. In summary, $\GL(V)_U$ is the
special orthogonal group
\[
  \GL(V)_U=\{g\in\GL(V)\mid gAg^\T=A\text{\ \ and\ \ $\det(g)=1$}\}.
\]
In particular, if $A=I$ and $W=U_I$, then 
$$
\GL(V)_W
=\{g\in\GL(V)\mid gg^\T=I\text{ and $\det(g)=1$}\}=\textup{SO}_3(\F).
$$

The quadratic form $Q_A\colon V\to\F\colon v\mapsto\frac12 vAv^\T$
determines (and is determined by) a non-degenerate symmetric
bilinear form $\beta_A\colon V\times V\to\F\colon(v,w)\mapsto vAw^\T$.
By diagonalizing $\beta_A$ (see \cite[Ch.~1,\,Cor.~1.2.4]{Lam}), there exists
$g_1\in\GL(V)$ such that $g_1^{-1}A(g_1^{-1})^\T$ is a non-zero scalar matrix.
Thus $(U_A)g_1=U_{g_1^{-1}A g_1^{-\T}\det(g_1)}=U_{\lambda I}$
where $\lambda I$ is a non-zero scalar matrix.
However, $(U_{\lambda I})(\lambda^{-1} I)=U_I$. Thus $U_Ag=U_I=W$
where $g=g_1\lambda^{-1}$.
\end{proof}

\noindent{\it The proof of Theorem~$\ref{r23}$(c).} 
Let $p$ be an odd prime, set $H=H_{p,3}$ and let 
$G=H/N$ be a 3-generator UCS $p$-group. By Lemma~\ref{strmodules}, 
the $\GL(\overline H)$-modules
$\overline H$ and $H^p$ are equivalent via the $p$-th power map.
The action of $\GL(\overline H)$ on $H'$ is equivalent to the exterior 
square action. In the proof of Lemma~\ref{solemma}, the action of $\GL(V)$ on
$\Lambda^2 V$ was shown to be $g\mapsto\det(g)(g^{-1})^\T$. Thus by
Lemma~\ref{strmodules} and Theorem~\ref{3th} the stabilizer
$K:=\GL(\overline H)_N$ acts irreducibly on $\overline H\cong H^p\cong V$
and on $H'\cong\Lambda^2 V$.

If $G$ has exponent $p$, then $H^p\leq N$.
By Lemma~\ref{strmodules}, $H'$ is invariant under $\GL(\overline H)$, 
and so 
the subspace $N\cap H'$ must be invariant under $\GL(\overline H)_N$. 
If
$1<N\cap H'<H'$, then $\GL(\overline H)_N$ 
is reducible on $H'$ contradicting the previous paragraph. Thus, 
by Theorem~\ref{3th}, $N\cap H'=1$, and 
$G=H/H^p$. Hence $G$ must 
be isomorphic to the group $G_3$ in the statement of the theorem,
and $\Aut(G_3)^{\overline{G_3}}\cong\GL_3(p)$.

Suppose now that $G$ has exponent $p^2$. By Theorem~\ref{3th}(iii), $|N|=p^3$, 
$N\cap H^p=1$ and $\GL(\overline H)_N$ is irreducible. As the 
$\GL(\overline H)$-actions on $H'$ and $\Lambda^2 \overline H$ are
equivalent,
Lemma~\ref{solemma} with $\F=\F_p$ shows that a minimal generating set $x_1,\ x_2,\ x_3$ of
$G$ can be chosen satisfying the relations of $G_4$. Lemma~\ref{solemma}
also yields that $\Aut(G_4)^{\overline{G_4}}\cong\SO_3(p)$. As $\SO_3(p)$ 
acts irreducibly on both $G/G^p$ and $G^p$, we see that
$G=G_4$ is a UCS-group, as required.
\end{proof}


\section{$4$-generator UCS 2-groups}\label{s4.5}


In this brief section we describe a computer-based 
classification of $4$-generator UCS $2$-groups. Recall that
$\{x_1,x_2,x_3,x_4\}$ is a fixed minimal generating set for $H_{2,4}$. 
Let $y_1,y_2,y_3,y_4$
denote the squares $x_1^2,x_2^2,x_3^2,x_4^2$,
and let $z_1,z_2,z_3,z_4,z_5,z_6$ denote the commutators
$[x_1,x_2],[x_1,x_3],[x_1,x_4],[x_2,x_3],[x_2,x_4],[x_3,x_4]$ in $H_{2,4}$, 
respectively. Each group below has the form
$H_{2,4}/N$ where $N$ is a subgroup of the Frattini subgroup
\[
  \Phi(H_{2,4})=
  \langle y_1,y_2,y_3,y_4,z_1,z_2,z_3,z_4,z_5,z_6\rangle.
\]
The following theorem proves the part of Theorem~\ref{main}(c) with $p=2$.
The homocyclic abelian group $H_{2,4}/N_5$ below is not included in
Theorem~\ref{main}(c).

\begin{theorem}\label{4th}
A $4$-generator UCS $2$-group is isomorphic to the group $H_{2,4}/N$
where $N$ is precisely one of the 9 subgroups described below:
\begin{align*}
N_1&=\langle y_1, y_2, y_3, y_4, z_1z_3, z_2, z_3z_4, z_5, z_6\rangle;
\hskip7.5mm N_2=\langle y_1, y_2y_3, y_3z_4, y_4, z_1z_3, z_2, z_3z_4, z_5, z_6\rangle;\\
N_3&=\langle y_1z_1, y_2z_1, y_3, y_4, z_1z_2z_3, z_2z_3z_5, z_3z_4, z_6\rangle;\\
N_4&=\langle y_1z_1, y_2z_2, y_3z_2, y_4z_1, z_1z_5, z_2z_3z_5, z_3z_4, z_6\rangle;\\
N_5&=\langle z_1, z_2, z_3, z_4, z_5, z_6\rangle;\hskip31.5mm
N_6=\langle y_1z_3, y_2, y_3, y_4, z_1z_5, z_6\rangle;\\
N_7&=\langle y_1z_2, y_2z_5, y_3, y_4, z_1z_6, z_2z_5z_6\rangle;\hskip13mm
N_8=\langle y_1z_2z_4, y_2y_4z_3, y_3y_4z_4, y_4z_1, z_1z_6, z_2z_5z_6\rangle;\\
N_9&=\langle y_1z_3, y_2z_4, y_3z_4, y_4z_3, z_1z_6, z_2z_5z_6\rangle.
\end{align*}
\end{theorem}

\begin{proof}
The proof relies primarily on computer calculations. One can
easily verify, using the command
{\tt CharacteristicSubgroups} in {\sf GAP}~\cite{GAP}, 
that each of the above 9 quotient groups are UCS
$2$-groups.  Clearly, they all are $4$-generator groups.

Suppose that $G$ is a $4$-generator UCS $2$-group. Then $G\cong H_{2,4}/N$
where $N\leq\Phi(H_{2,4})$ by Lemma~\ref{idgh}. If $G$ is abelian, then 
$G^2$ is the non-trivial and proper characteristic subgroup of $G$. In this
case, $G$ must be isomorphic to the homocyclic group $H_{2,4}/N_5\cong(C_4)^4$.

Assume now that $G$ is non-abelian. 
As $\Phi(G)=G'$ and
$|(H_{2,4})'|=2^6$, we deduce that $|G|\leq 2^{10}$. The classification of
2-groups with order at most $2^{10}$ is part of the computational algebra
systems {\sc Magma} and {\sf GAP}~\cite{Magma,GAP}.
Both systems can be used to determine whether a given 2-group is a
$4$-generator UCS group. Although the system {\sc Magma} was
faster then {\sf GAP}, it was still unable to deal with the large number
of groups of order
$2^{10}$. Suppose that $H$ is a $4$-generator group in the {\sc Magma}
library with order dividing $2^9$ where $H'=\Phi(H)=Z(H)$. We used {\sc Magma}
to compute $\Aut(H)$, and checked whether $\Aut(H)^{\overline H}$ and 
$\Aut(H)^{\Phi(H)}$ are irreducible linear groups. By Theorem~\ref{3th}, 
$H$ is UCS if and only if both of these linear groups are irreducible.
In this way one can verify that if $|G|$ divides $2^9$, then
$G$ is precisely one of the 9 quotient groups above. Looping over the 
groups with order $2^9$ on a computer with a 1.8 GHz CPU and
512 MB memory took approximately 20 CPU minutes. A similar
computation for $|G|=2^{10}$ never completed.

It remains to show that no 4-generator UCS group $G=H_{2,4}/N$ has order
$2^{10}$ where $N\leq\Phi(H_{2,4})$ has order $|N|=2^4$.
By Theorem~\ref{3th}, $S:=\GL_4(2)_N$ is an irreducible subgroup of
$\GL_4(2)$. View $N$ as a subspace of 
$\Phi(H_{2,4})\cong(\F_2)^{10}$, and identify $\Phi(H_{2,4})/(H_{2,4})'$ with
$V=(\F_2)^4$, and $(H_{2,4})'$ with $\Lambda^2 V$.
As $|G'|=|\Lambda^2 V|=2^6$, we see that $N\cap\Lambda^2 V=0$,
and so $\Phi(H_{2,4})$ admits the $S$-module direct
decomposition $\Phi(H_{2,4})=N\oplus\Lambda^2 V$. Further, since 
$\Phi(H_{2,4})/N\cong \Lambda^2 V$, the action of $S$ on $\Lambda^2 V$ is
irreducible by Theorem~\ref{3th}. We use {\sc Magma} to loop over
all irreducible subgroups
$R$ of $\GL(V)\cong\GL_4(2)$. We find that either $R$ is reducible on
$\Lambda^2 V$, or the $R$-action on $\Phi(H_{2,4})$ fixes no
$4$-dimensional subspace that could correspond to $N$. This contradiction
proves that no such group $G$ exists.
\end{proof}

The {\sf GAP} and {\sc Magma} catalogue 
numbers of the 9 quotient groups in Theorem~\ref{4th} are
$[ 32, 49 ]$, $[ 32, 50 ]$,
$[ 64, 242 ]$, $[ 64, 245 ]$,
$[ 256, 6732 ]$,
$[ 256, 8935 ]$, $[ 256, 10090 ]$, $[ 256, 10289 ]$, and
$[ 256, 10297 ]$, respectively. 
The 1-dimensional semilinear group $\Gamma \textup{L}_1(16)$
has a presentation $\langle a,b\mid a^4=b^{15}=1,b^a=b^2\rangle$.
The automorphism groups of the 9 USC 2-groups
induce on the Frattini quotient the following irreducible
subgroups of $\GL_4(2)$:
$\textup{O}^{+}_4(2)$, $\textup{O}^{-}_4(2)$,
$\GL_2(2)\boxtimes \GL_2(2)\cong S_3\times S_3$, 
$\Gamma \textup{L}_1(16)$,
$\GL_4(2)$, $\textup{O}^{+}_4(2)$, $\langle a,b^3\rangle$,
$\langle b^3\rangle$, and
$\Gamma \textup{L}_2(4)$ respectively.
This can be deduced by using {\sf GAP} or {\sc Magma} to compare
the chief series of irreducible subgroups of $\GL_4(2)$, and 
the chief series of automorphism groups of USC 2-groups.

The groups in Theorem~\ref{4th} indicate that an important class
of UCS $2$-groups are formed by the Suzuki 2-groups. 
Computation with GAP~\cite{GAP} shows that the group 
$H_{2,4}/N_4$  is isomorphic to 
the Suzuki 2-group $\mathcal Q$ described as a $(3\times 3)$-matrix 
group in~\cite[VIII.7.10~Remark]{HB} with $n=4$. Further, the group
$H_{2,4}/N_9$ is isomorphic to the group $A(4,\vartheta)$ 
in~\cite[6.7~Example]{HB} where
$\vartheta$ is the Frobenius automorphism of $\F_{16}$. However, 
since $\vartheta$ has order $4$, by~\cite[6.9~Theorem]{HB},
this group is not a Suzuki 2-group. 
In a
Suzuki 2-group $G$, the subgroup $\Phi(G)$ coincides
with the elements of order at most~2~\cite[VII.7.9~Theorem]{HB}, 
and so the automorphism $\xi$ that permutes the 
involutions transitively acts irreducibly on $\Phi(G)$. 
Moreover, it is noted in 
the proof of~\cite[7.9~Theorem(b)]{HB} that in a Suzuki 2-group $G$ 
of type
$A(\vartheta,n)$ the automorphism $\xi$ that permutes the set of involutions
transitively acts irreducibly on $\overline G$, and so Theorem~\ref{3th} 
implies that these groups are always UCS groups. We claim, in addition,
that the group
$\mathcal Q$ in~\cite[VIII.7.10~Remark]{HB} is always UCS. Indeed,
the group of diagonal matrices with $z^{1-q},z,z^2$ in the diagonal
where $z$ runs through the elements of $\F_{q^2}$ induces a Singer
cycle on the quotient $\mathcal Q/\overline{\mathcal Q}$, and hence this
group is irreducible on $\mathcal Q/\overline{\mathcal Q}$. 
Thus Theorem~\ref{3th} gives that $\mathcal Q$ is UCS. Though 
we conjecture that Suzuki 2-groups are always UCS groups, the detailed 
investigation of these groups goes beyond the scope of this paper.




\section{$4$-generator UCS $p$-groups with exponent $p$}\label{s5}

Suppose that $p$ is odd. 
In this section we classify $4$-generator UCS $p$-groups with
exponent $p$. That is, we prove the following theorem.

\begin{theorem}\label{4gen}
Let $p$ be an odd prime, let $\{x_1,x_2,x_3,x_4\}$ be a minimal 
generating set for $H_{p,4}$. Fix $\alpha\in\F_p^\times$
such that $\F_p^\times=\langle-\alpha\rangle$. 
The following is a complete and irredundant 
list of the isomorphism classes
of $4$-generator UCS $p$-groups with exponent $p$:
\begin{enumerate}
\item[(i)] $G_0=H_{p,4}/(H_{p,4})^p$;
\item[(ii)] $G_2=H_{p,4}/\left<(H_{p,4})^p,
[x_1,x_2][x_3,x_4]^{-1},[x_1,x_3],[x_1,x_4],[x_2,x_3],[x_2,x_4]\right>$;
\item[(iii)] $G_4=H_{p,4}/\left<(H_{p,4})^p,[x_1,x_3],[x_1,x_4],[x_2,x_3],[x_2,x_4]\right>$;
\item[(iv)] $G_6=H_{p,4}/\left<(H_{p,4})^p,[x_1,x_2],[x_3,x_4],[x_2,x_3][x_1,x_4]^{-1},[x_1,x_3]^\alpha[x_2,x_4]\right>$;
\item[(v)] $G_{11}=H_{p,4}/\left<(H_{p,4})^p,[x_1,x_4],[x_2,x_3],[x_2,x_4][x_1,x_3]^{-1}\right>$;
\item[(vi)] $G_{14}=H_{p,4}/\left<(H_{p,4})^p,[x_1,x_2][x_3,x_4]\right>$;
\item[(vii)] $G_{16}=H_{p,4}/\left<(H_{p,4})^p,[x_1,x_2],[x_3,x_4]\right>$;
\item[(viii)] $G_{18}=H_{p,4}/\left<(H_{p,4})^p,[x_2,x_3][x_1,x_4],[x_1,x_3]^\alpha[x_2,x_4]^{-1}\right>$.
\end{enumerate}
\end{theorem}

The indexing of groups in Theorem~\ref{4gen} is explained later. It
is related to the 18 orbits of $\GL(V)$ on the proper non-trivial
subspaces of $\Lambda^2 V$.

We recall that $H_{p,r}$ is defined in Section~\ref{s3}.
Theorem~\ref{4gen} is the quantitative version of Theorem~\ref{main}(c)
and its proof relies on the classification of $4$-generator 
$p$-groups with exponent $p$ and nilpotency class~$2$. 
In order to present the classification here,
we need some notation.
Since $p$ is odd,
a $4$-generator $p$-group $G$ with nilpotency 
class $2$ and exponent $p$ is isomorphic to
$H_{p,4}/((H_{p,4})^pN)$ where 
$N$ is a subgroup of $(H_{p,4})'$. Since $(H_{p,4})'$ is an elementary abelian 
group, we view it as a vector space over $\F_p$. 
Assume that $H_{p,4}$ is generated by $x_1,\ x_2,\ x_3,\ x_4$. 
Consider the subgroup $(H_{p,4})'$ as a 6-dimensional subspace over $\F_p$ 
with standard basis $[x_1,x_2]$, $[x_1,x_3]$, $[x_1,x_4]$, $[x_2,x_3]$,
$[x_2,x_4]$, $[x_3,x_4]$. We introduce a non-degenerate 
quadratic form on $(H_{p,4})'$:
$$
Q'\left([x_1,x_2]^{\alpha_1}
[x_1,x_3]^{\alpha_2}[x_1,x_4]^{\alpha_3}[x_2,x_3]^{\alpha_4}
[x_2,x_4]^{\alpha_5}[x_3,x_4]^{\alpha_6}\right)
=\alpha_1\alpha_6-\alpha_2\alpha_5+\alpha_3\alpha_4.
$$
The quadratic form $Q'$ induces a non-degenerate
symmetric bilinear form:
$$
(v_1,v_2)=Q'(v_1+v_2)-Q'(v_1)-Q'(v_2).
$$
If $U$ is a subgroup in $(H_{p,4})'$, then $U^\perp$
is defined as 
$$
U^\perp=\{v\in (H_{p,4})'\ |\ (u,v)=0\mbox{ for all }u\in U\}.
$$
Let $\alpha$ be as in Theorem~\ref{4gen}, and
define the following 
subgroups in $(H_{p,4})'$:
\begin{itemize}
\item $N_0=1$;
\item $N_1=\left<[x_1,x_3],\,[x_1,x_4],\,[x_2,x_3],\,[x_2,x_4],\,[x_3,x_4]\right>$;
\item $N_2=\left<[x_1,x_2][x_3,x_4]^{-1},\,[x_1,x_3],\,[x_1,x_4],\,[x_2,x_3],\,[x_2,x_4]\right>$;
\item $N_3=\left<[x_1,x_2],\,[x_1,x_4],\,[x_2,x_4],\,[x_3,x_4]\right>$;
\item $N_4=\left<[x_1,x_3],\,[x_1,x_4],\,[x_2,x_3],\,[x_2,x_4]\right>$;
\item $N_5=\left<[x_1,x_2],\,[x_2,x_4],\,[x_3,x_4],\,[x_2,x_3][x_1,x_4]^{-1}\right>$;
\item $N_6=\left<[x_1,x_2],\,[x_3,x_4],\,[x_2,x_3][x_1,x_4]^{-1},\,[x_1,x_3]^\alpha[x_2,x_4]\right>$;
\item $N_7=\left<[x_1,x_4],\,[x_2,x_4],\,[x_3,x_4]\right>$;
\item $N_8=\left<[x_1,x_3],\,[x_1,x_4],\,[x_2,x_4]\right>$;
\item $N_9=\left<[x_2,x_3],\,[x_2,x_4],\,[x_3,x_4]\right>$;
\item $N_{10}=\left<[x_1,x_3],\,[x_1,x_4],\,[x_3,x_4][x_1,x_2]^{-1}\right>$;
\item $N_{11}=\left<[x_1,x_4],\,[x_2,x_3],\,[x_2,x_4][x_1,x_3]^{-1}\right>$;
\item $N_{12}=\left<[x_1,x_4],\,[x_2,x_4][x_1,x_3]^{-\alpha},\,[x_1,x_2][x_3,x_4]^{-1}\right>$.
\end{itemize}
For $i=13,\ldots,18$ set $N_i=(N_{i-12})^\perp$. That is, $N_i$ is
the subgroup perpendicular to $N_{i-12}$ with respect to the
symmetric bilinear form $(\;,\;)$ associated to $Q'$.
For $i=0,\ldots,18$, let $G_i$ denote the group $H_{p,4}/((H_{p,4})^pN_i)$.
Our notation is consistent in the sense that the groups $G_i$
defined here coincide with those defined in Theorem~\ref{4gen}.

\begin{lemma}\label{dual}
For $p\geq 3$, every $4$-generator finite $p$-group with
exponent $p$ and nilpotency class~$2$ is isomorphic to precisely one of the
groups $G_0,G_1,\ldots,G_{18}$.
\end{lemma}

\begin{proof}
Suppose that $i\in\{0,\ldots,5\}$ and let $\mathcal N_i$ denote the set of 
subgroups with order $p^{i}$ in $(H_{p,4})'$. 
Then, for $M_1,\ M_2\in\mathcal
N_i$, we have $H_{p,4}/((H_{p,4})^pM_1)\cong H_{p,4}/((H_{p,4})^pM_2)$
if and only if $M_1$ and $M_2$ lie in the same $\GL(\overline{H_{p,4}})$-orbit
(see~\cite[Theorem~2.8]{obrien}). Therefore it suffices to show that the
subgroups $N_0,\ldots,N_{18}$ form a complete and irredundant set of 
representatives of the $\GL(\overline{H_{p,4}})$-orbits in $\bigcup
\mathcal N_i$.
A simple computation shows that $Q'$ is stabilized by
$\GL(\overline{H_{p,4}})$ up to scalar multiples
(alternatively see~\cite[Proposition~2.9.1(vii)]{kl}).
Thus $M_1$ and $M_2$ lie in the same $\GL(\overline{H_{p,4}})$-orbit
if and only if
$(M_1)^\perp$ and $(M_2)^\perp$ do. 
Therefore we only need to verify that the set $\{N_1,\ldots,N_{12}\}$ 
is a set of representatives of the $\GL(\overline{H_{p,4}})$-orbits in 
$\mathcal N_5\cup\mathcal N_4\cup\mathcal N_3$. These orbits have long been
known; they are listed already in Brahana's paper~\cite{brahana}. 
Our list above is taken 
from the recent classification of finite $p$-groups of order dividing $p^7$,
see~\cite{p6,p7}.
\end{proof}

Recall that $\{x_1,x_2,x_3,x_4\}$ is a fixed generating set for $H_{p,4}$. 
Risking confusion, let  $V$ denote a 4-dimensional vector space 
over $\F_p$ with basis $x_1,x_2,x_3,x_4$. This way the symbol $x_i$ 
may refer to an element of $H_{p,4}$, or to an element of $V$.
However, elements of the group $H_{p,4}$ are written multiplicatively,
while elements of $V$ are written additively.
There is a natural bijection 
$\Psi:\Lambda^2 V\rightarrow (H_{p,4})'$ 
mapping $x_i\wedge x_j\mapsto [x_i,x_j]$. Define a quadratic form 
$Q$ on $\Lambda^2 V$ by $Q(w)=Q'(\Psi(w))$ for all $w\in \Lambda^2 V$.
The value of the form $Q$ is given by:
\begin{multline*}
Q\left(\alpha_1x_1\wedge x_2+
\alpha_2x_1\w x_3+\alpha_3x_1\w x_4+\alpha_4x_2\w x_3+
\alpha_5x_2\w x_4+\alpha_6x_3\w x_4\right)\\
=\alpha_1\alpha_6-\alpha_2\alpha_5+\alpha_3\alpha_4.
\end{multline*}

Theorem~\ref{3th} says that a subspace $U\leq\Lambda^2 V$ gives rise to
an exponent-$p$ UCS $p$-group
$G_U:=H_{p,4}/((H_{p,4})^p\Psi(U))$ if and only
if the stabilizer $\GL(V)_U$
is irreducible on both $V$ and $\Lambda^2 V/U$. For $i=0,\ldots,18$,
let $U_i$ denote the subspace $\Psi^{-1}(N_i)$. Therefore we can check
which of the groups
$G_0,\ldots,G_{18}$ are UCS by checking, for $i=0,\ldots,18$, 
whether $\GL(V)_{U_i}$ is irreducible  on $V$ and on $\Lambda^2 V/U_i$.
This is carried out in the rest of this section.

A subspace $U\leq\Lambda^2 V$ is said to be \emph{degenerate} if
$U\cap U^\perp\neq 0$; otherwise it is said to be \emph{non-degenerate}.
A subspace $U$ is said to be \emph{totally isotropic} if $Q(u)=0$
for all $u\in U$.

\begin{lemma}\label{notucs}
If $U$ is a degenerate subspace of $\VwedgeV$, then $\GL(V)_U$ acts reducibly
on $V$ or $\VwedgeV/U$, and thus the $p$-group
$G_U:=H_{p,4}/((H_{p,4})^p\Psi(U))$ is not
a UCS-group. Further, the subspaces $U_i$ for
$i\in\{1,3,5,7,8,9,10,12,13,15,17\}$ are degenerate with respect to $Q$.
\end{lemma}

\begin{proof}
Suppose that $U$ is degenerate. Then $U\cap U^\perp\ne0$, and $U+U^\perp$
is a proper subspace of $\Lambda^2 V$. Thus if $U^\perp\not\leq U$, then
$\GL(V)_U=\GL(V)_{U^\perp}$ is reducible on $\Lambda^2 V/U$ as it
fixes the proper subspace $(U+U^\perp)/U$.
Assume henceforth that $U^\perp\leq U$. 

Lemma~\ref{dual} says that there are
precisely 19 different
$\GL(V)$-orbits on the proper subspaces of $\VwedgeV$, and representatives
of the orbits are $U_0,U_1,\dots,U_{18}$. 
A straightforward calculation
shows that $U_i$ is degenerate if and only if $i\in \{1,3,5,7,8,9,10,12,13,15,17\}$. Moreover, the only $U_i$
satisfying $U^\perp\leq U$ are $U_1,U_3,U_7,U_9$. It is possible to prove case
by case that for these values of $i$, the group $G_{U_i}$ is not UCS.
However, instead of performing a case by case analysis, we offer a geometric
proof, which we find more elegant.

We shall show that for each degenerate subspace
$U\leq\Lambda^2V$ with $U^\perp\leq U$, the stabilizer $\GL(V)_{U^\perp}=\GL(V)_U$
is a reducible subgroup of $\GL(V)$. 
Since $\dim(U^\perp)=6-\dim(U)$, we see that
$\dim(U^\perp)=1,2,3$
and $U^\perp$ is a totally isotropic subspace of $\Lambda^2 V$.
First, suppose that $\dim(U^\perp)=1$.
By \cite[p.\,187]{taylor}, there exist linearly independent vectors
$v_1,v_2\in V$ such that $U^\perp=\langle v_1\w v_2\rangle$ and hence
$\GL(V)_{U^\perp}$ fixes $\langle v_1,v_2\rangle$ and so is reducible.
Second, suppose that $\dim(U^\perp)=2$.
Then also by \cite[Lemma~12.15]{taylor}, $U^\perp$
has the form $\langle v_1\w v_3,v_2\w v_3\rangle$ where
$v_1,v_2,v_3\in V$ are linearly independent. Thus
$\GL(V)_{U^\perp}$ fixes $\langle v_1,v_2,v_3\rangle$.
Finally, suppose that $\dim(U^\perp)=3$,
then by \cite[Theorem~12.16]{taylor}, $U^\perp$ equals
$W_1\wedge V$ or $W_3\wedge W_3$ where $W_1,W_3$ are subspaces of $V$
of dimension 1 and 3 respectively. Thus $\GL(V)_{U^\perp}$ fixes
$W_1$ or $W_3$, and so is reducible.
\end{proof}

The proof of the converse of Lemma~\ref{notucs} is substantially harder.
It is noteworthy that in the proof of Lemma~\ref{goodU} below,
$\Aut(G)^{\overline G}$ is commonly a maximal group from one of
Aschbacher's~\cite{A84} classes.

\begin{lemma}\label{goodU}
If $U$ is a non-degenerate subspace of $\VwedgeV$, then $\GL(V)_U$ acts
irreducibly on $V$ and $\VwedgeV/U$, and thus the $p$-group
$G_U:=H_{p,4}/((H_{p,4})^p\Psi(U))$ is a UCS-group. Further,
the subspaces $U_i$ for $i\in\{0,2,4,6,11,14,16,18\}$ are non-degenerate
with respect to $Q$.
\end{lemma}

\begin{proof}
As remarked in the proof of Lemma~\ref{notucs}, $U_0,U_1,\dots,U_{18}$
is a complete and irredundant list of representatives of
$\GL(V)$-orbits on the proper subspaces of $\VwedgeV$.
Moreover, $U_i$ is non-degenerate if and only if
$i\in\{0,2,4,6,11,14,16,18\}$.
Denote the stabilizer $\GL(V)_{U_i}$ by $K_i$. 
We prove below, in a case by case manner, that $K_i$
acts irreducibly on $V$ and $\VwedgeV/U_i$ for $i\in\{0,2,4,6,11,14,16,18\}$.

Since $\GL(V)$ is irreducible on $V$ and $\Lambda^2V$, $G_0$ is a UCS-group.
Suppose henceforth that $U_i\neq0$. For non-degenerate $U$,
$\VwedgeV=U\oplus U^\perp$ and $\GL(V)_U=\GL(V)_{U^\perp}$. Since
$U_{i+12}=U_i^\perp$ for $1\leq i\leq6$, we see that $K_i=K_{i+12}$. Thus it suffices
to prove that $K_i$ is irreducible on the spaces $V$, $U_i$, and
$U_i^\perp\cong\Lambda^2 V/U_i$ for $i\in\{2,4,6,11\}$.

Consider first the stabilizer $K_2$. Note that
$U_2^\perp=U_{14}=\left<x_1\w x_2+x_3\w x_4\right>$. View $g\in\GL(V)$ 
as a $4\times 4$ matrix with respect to the basis $x_1,x_2,x_3,x_4$ of $V$.
The equation
\[
  (x_1\w x_2+x_3\w x_4)(g\w g)=\alpha(x_1\w x_2+x_3\w x_4)
  \quad\text{for some non-zero $\alpha\in\F_p$},
\]
gives a linear system equivalent to the matrix equation $g^TJg=\alpha J$ where
$$
J=\left(\begin{array}{cccc}
0 & 1 & 0 & 0 \\
-1 & 0 & 0 & 0 \\
0 & 0 & 0 & 1 \\
0 & 0 & -1 & 0
\end{array}\right).
$$
Thus $g\in K_2$ if and only if $g$ preserves the alternating
form $J$ up to a scalar factor. In other words,
$K_2\cong\GSp_4(p)$. Clearly, $\GSp_4(p)$ is irreducible on $V$,
and on the $1$-dimensional subspace $U_{14}$. Additionally, 
the $K_2$-action on $U_2$ is irreducible as $K_2$
contains a subgroup isomorphic 
to $\Omega_5(p)$  (see \cite[Proposition~2.9.1(vi)]{kl}),
and $\Omega_5(p)$ acts irreducibly on the 5-dimensional space $U_2$.

Consider now the stabilizer $K_4$. Set $L_1=\left<x_1,x_2\right>$ and
$L_2=\left<x_3,x_4\right>$. Let $H$ be the stabilizer of the
decomposition $V=L_1\oplus L_2$. Then $H\cong\GL_2(p)\wr C_2$. 
We aim to prove that $H=K_4$.
It is routine to check that $H$ stabilizes
$U_{16}=\left<x_1\w x_2,x_3\w x_4\right>$ and $U_{16}^\perp=U_4$. 
Thus $H\leq K_4$. Set $P_1=\left<x_1\w x_2\right>$ and
$P_2=\left<x_3\w x_4\right>$. Then $U_{16}=P_1\oplus P_2$ and $P_1,P_2$
are the only 1-dimensional totally isotopic subspaces of $U_{16}$.
Thus $K_4$ permutes the set $\{P_1,P_2\}$. The following argument
shows that $\GL(V)_{P_1}=\GL(V)_{L_1}$:
\begin{align*}
  g\in\GL(V)_{P_1}&\Longleftrightarrow 
    (x_1\wedge x_2)(g\wedge g)=\alpha(x_1\wedge x_2)
    \quad\text{for some}\quad\alpha\in\F_p^\times\\
  &\Longleftrightarrow \langle x_1,x_2\rangle g=\langle x_1,x_2\rangle\\
  &\Longleftrightarrow g\in\GL(V)_{L_1}.
\end{align*}
Denote by $\widetilde{H}$ the subgroup of $H$ that fixes both
$L_1$ and $L_2$, and denote by $\widetilde{K_4}$ the subgroup of $K_4$
that fixes both $P_1$ and $P_2$. Since
$\GL(V)_{L_1}=\GL(V)_{P_1}$, we see that $\widetilde{H}=\widetilde{K_4}$.
Since $|K_4:\widetilde{K_4}|=2$ and $|H:\widetilde{H}|=2$, it follows 
that $|H|=|K_4|$. Thus $H=K_4$ as desired.

We claim that $K_4$ is irreducible on $V$, $U_4$, and on $U_{16}$.
The only proper and non-trivial $\overline K_4$-submodules
of $V$ are $L_1$ and $L_2$. These are swapped by $K_4$, and so
$K_4$ is irreducible on
$V$. Suppose that $g\in \overline K_4$ is represented by the block-matrix
$$
\left(\begin{array}{cc}
A & 0 \\
0 & B\end{array}\right).
$$
Then the action of $g$ on $U_{16}$ relative to the basis
$x_1\w x_2,x_3\w x_4$ has matrix
$$
\left(\begin{array}{cc}
\det A & 0 \\
0 & \det B\end{array}\right).
$$
Therefore the only proper, non-trivial $\overline K_4$-submodules of
$U_{16}$ are $P_1$ and $P_2$. However, $K_4$ contains an element
that swaps $P_1$ and $P_2$, and thus $K_4$ is irreducible on $U_{16}$. 
Recall that $U_4=\left<x_1\w x_3,x_1\w x_4,x_2\w x_3,x_2\w x_4\right>$.
The action of $\overline K_4$ on $U_4$ is equivalent to the action 
of $\GL_2(p)\times \GL_2(p)$ on the tensor product
$\left<x_1,x_2\right>\otimes\left<x_3,x_4\right>$;
the equivalence is realized by the map $x_i\w x_j\mapsto x_i\otimes x_j$.
Since $\GL_2(p)$ is absolutely irreducible, \cite[8.4.2]{rob} can be used to
obtain that the outer tensor product
$\GL_2(p)\boxtimes \GL_2(p)$ acts irreducibly on $\F_p^4$.
Hence $\overline K_4$, and therefore $K_4$, acts irreducibly on $U_4$. 

The next stabilizer to consider is $K_6$. We shall show that $K_6$ acts
irreducibly on $V$, $U_6$ and $U_{18}=U_6^\perp$. The definition of the
subspace $U_6$ involves a fixed $\alpha\in\F_p^\times$ such that $-\alpha$
generates $\F_p^\times$. Note that
\begin{equation}\label{U18}
  U_{18}=U_6^\perp=\left<x_1\w x_4+x_2\w x_3,\alpha x_1\w x_3-x_2\w x_4\right>.
\end{equation}
We shall prove that $K_6=K_{18}\cong\GammaL_2(p^2)$. 
Identify $V$ with $\F_{p^2}\oplus\F_{p^2}$ 
as follows. Since $-\alpha$ is a non-square in $\F_p$, its square roots,
$\pm\beta$, lie in $\F_{p^2}\setminus\F_p$. 
Then $(1,0),(\beta,0),(0,1),(0,\beta)$ is an $\F_p$-basis of 
$\F_{p^2}\oplus\F_{p^2}$. 
Identify the elements 
$x_1$, $x_2$, 
$x_3$, $x_4$ with $(1,0)$, $(\beta,0)$, $(0,1)$, $(0,\beta)$, respectively.
We view $\GammaL_2(p^2)$
as a subgroup of $\GL_4(p)$ under the identification above.

Let $\varepsilon:\Lambda^2_{\F_p}V\rightarrow \Lambda^2_{\F_{p^2}}V$ be the unique
$\F_p$-linear map satisfying $\varepsilon(u\wedge v)=u\wedge v$. 
Easy computation shows that $U_6\leq \ker\varepsilon$. On the other
hand, by dimension counting, we have that $\dim(\ker\varepsilon) =4$, which
gives $U_6=\ker \varepsilon$. 
Suppose that $U$ is a 2-dimensional $\F_p$-subspace in $V$. 
Then $U$ is an 
$\F_{p^2}$-subspace of $V$ if and only if $\varepsilon(U\wedge U)=0$; that is,
$U\wedge U\leq U_6$. Thus $K_6$ is the setwise stabilizer of 
the $\F_{p^2}$-subspaces of $V$. 

Clearly, the group $\GammaL_2(p^2)$ permutes the
$\F_{p^2}$ subspaces of $V$, and hence $\GammaL_2(p^2)\leq K_6$. 
In order to show the other direction, 
let $g\in\GL_4(p)$ and suppose that $g$ stabilizes $U_6$. 
Then $g$ permutes the
$\F_{p^2}$-subspaces. Thus we have, for $v\in V$ and $\alpha\in\F_{p^2}$, that 
\begin{equation}\label{beta}
(\alpha v)g=\beta(v g)
\end{equation}
 with some $\beta\in\F_{p^2}$. Let $v_1,\ v_2\in V$ be 
linearly independent over $\F_{p^2}$. Then, as $g$ permutes the $\F_{p^2}$-subspaces, 
$v_1g$ and $v_2g$ are linearly
independent over $\F_{p^2}$. Let $\beta_1,\ \beta_2\in\F_{p^2}$
such that  $(\alpha v_1)g=\beta_1 (v_1g)$ and
$(\alpha v_2)g=\beta_2(v_2g)$. Then there is some $\gamma\in\F_{p^2}$ such that
$(\alpha v_1+\alpha v_2)g=\gamma(v_1g+v_2g)$, 
but also $(\alpha v_1+\alpha v_2)g=\beta_1 (v_1g)+\beta_2 (v_2g)$. 
This shows that
$\beta_1=\beta_2$, and therefore $\beta$ is independent of $v$ in 
equation~\eqref{beta}. Hence for all $\alpha\in\F_{p^2}$ there
is some $\varphi(\alpha)\in\F_{p^2}$ such that $(\alpha v)g=\varphi(\alpha) (vg)$. As 
$$
\varphi(\alpha_1+\alpha_2)(vg)=((\alpha_1+\alpha_2)v)g=(\alpha_1v+\alpha_2v)g=(\varphi(\alpha_1)+\varphi(\alpha_2))(vg),
$$ 
we obtain that $\varphi$ is additive. Since
$\varphi(\alpha\beta) vg=
(\alpha\beta v)g=\varphi(\alpha)((\beta v)g)=\varphi(\alpha)\varphi(\beta)(vg)$,
we have that $\varphi$ is a field automorphism. Since $\varphi$ fixes $\F_p$
pointwise, $\varphi$ is a member of the Galois group of $\F_{p^2}$ over
$\F_p$. Hence $g$ is a semilinear transformation which 
gives that $K_6\leq\GammaL_2(p^2)$. Therefore $K_6=\GammaL_2(p^2)$, as claimed.

To show that $K_6$ acts irreducibly on $U_6$ and $U_{18}$ let us take a 
Singer cycle, i.e., an element $g$ of order $p^4-1$ in $\GL_2(p^2) < K_6$. 
Let $\varepsilon\in \F_{p^4}$ be an eigenvalue of $g$. 
Then the eigenvalues of $g$ are $\varepsilon$, $\varepsilon^p$, 
$\varepsilon^{p^2}$, 
$\varepsilon^{p^3}$, and so the eigenvalues of $g\wedge g$ on $V\wedge V$ are 
$\eta=\varepsilon^{1+p}$, $\eta^p$, $\eta^{p^2}$, $\eta^{p^3}$,
$\theta=\varepsilon^{1+p^2}$, and $\theta^p$. The order of $\eta$ is 
$(p^4-1)/(p+1)>p^2$, hence $\eta$, $\eta^p$, $\eta^{p^2}$, $\eta^{p^3}$ 
are all different. Similarly, the order of $\theta$ is $p^2-1$, 
so $\theta$ and $\theta^p$ are distinct. This means that the characteristic 
polynomial of $g\wedge g$ is the product of two irreducible factors, one of 
degree 4, the another of degree 2. Hence $V\wedge V$ decomposes into a direct 
sum of a 4-dimensional and a 2-dimensional irreducible $\langle g\wedge g 
\rangle$-submodules. Since $U_6$ and $U_{18}$ are invariant under 
$g\wedge g$, we obtain that these are the irreducible summands.

Finally, consider $K_{11}$.
We identify $V$ with 
the tensor product $\left<u_1,u_2\right>\otimes\left<v_1,v_2\right>$ by 
assigning $x_1$, 
$x_2$, $x_3$, $x_4$ to $-u_1\otimes v_1$, $u_2\otimes v_1$, 
$u_2\otimes v_2$, $u_1\otimes v_2$, respectively. Hence the group
$\GL_2(p)\times\GL_2(p)$ acts on $V$, and the action of 
the element 
$$
\left(\begin{array}{cc}
\alpha_{11} & \alpha_{12}\\
\alpha_{21} & \alpha_{22}
\end{array}\right)
\otimes
\left(\begin{array}{cc}
\beta_{11} & \beta_{12}\\
\beta_{21} & \beta_{22}
\end{array}\right)
$$
is represented by the matrix
\begin{equation}\label{tensor}
\left(\begin{array}{cccc}
\alpha_{11}\beta_{11} & -\alpha_{12}\beta_{11} & -\alpha_{12}\beta_{12} & -\alpha_{11}\beta_{12}\\
-\alpha_{21}\beta_{11} & \alpha_{22}\beta_{11} & \alpha_{22}\beta_{12} & \alpha_{21}\beta_{12}\\
-\alpha_{21}\beta_{21} & \alpha_{22}\beta_{21} & \alpha_{22}\beta_{22} & \alpha_{21}\beta_{22}\\
-\alpha_{11}\beta_{21} & \alpha_{12}\beta_{21} & \alpha_{12}\beta_{22} & \alpha_{11}\beta_{22}
\end{array}\right).
\end{equation}
The kernel of the action of $\GL_2(p)\times\GL_2(p)$ on $V$ equals
$\{\lambda I\otimes \lambda^{-1}I\mid\lambda\in\F_p^\times\}$.
Thus the central product
$\GL_2(p)\Y\GL_2(p)$ acts faithfully on $V$.
Let $H$ denote the group of all matrices
of the form~\eqref{tensor}.
Elementary, but cumbersome, calculation shows that the group 
$H$ is the stabilizer of 
$U_{11}$, and so $H=K_{11}$. 
In particular 
$K_{11}\cong \GL_2(p)\Y\GL_2(p)$. As $\GL_2(p)$ is absolutely irreducible
on $\F_p^2$, by~\cite[8.4.2]{rob},  $K_{11}$ is irreducible on $V$.
 
Consider the subgroups $T_1$ and $T_2$ of $K_{11}$ 
consisting of the elements of 
the form
$$
\left(\begin{array}{cc}
\alpha_{11} & \alpha_{12}\\
\alpha_{21} & \alpha_{22}
\end{array}\right)
\otimes
\left(\begin{array}{cc}
1 & 0\\
0 & 1
\end{array}\right)\quad\mbox{and}\quad\left(\begin{array}{cc}
1 & 0\\
0 & 1
\end{array}\right)
\otimes
\left(\begin{array}{cc}
\beta_{11} & \beta_{12}\\
\beta_{21} & \beta_{22}
\end{array}\right).
$$
Simple computation shows that the action of a generic element of 
$T_1$ on $U_{11}$ is represented 
by the matrix
$$
\left(
\begin{array}{ccc}\alpha_{11}^2 & -\alpha_{12}^2 & -\alpha_{12}\alpha_{11}\\
-\alpha_{21}^2 & \alpha_{22}^2 & \alpha_{22}\alpha_{21}\\
-2\alpha_{21}\alpha_{11} & 2\alpha_{22}\alpha_{12} & \alpha_{22}\alpha_{11}+\alpha_{12}\alpha_{21}
\end{array}\right).
$$
Since $T_1$ is
isomorphic to $\GL_2(p)$, the derived subgroup $(T_1)'$ is isomorphic to 
$\SL_2(p)$ and simple computation shows that $(T_1)'$ induces a subgroup 
of $\SL(U_{11})$. Moreover, $(T_1)'$ preserves the symmetric bilinear form
\begin{multline*}
(\alpha_1 x_1\wedge x_4+\alpha_2x_2\wedge x_3+\alpha_3(x_2\wedge x_4-x_1\wedge
  x_3),\beta_1x_1\wedge x_4+\beta_2x_2\wedge x_3+\beta_3(x_2\wedge
  x_4-x_1\wedge x_3))\\=
-\alpha_1\beta_2-\alpha_2\beta_1-2\alpha_3\beta_3.
\end{multline*}
It is shown in~\cite[Proposition~2.9.1(ii)]{kl} that $(T_1)'$ induces  $\Omega(\hat Q)$,
where $\hat Q$ is the quadratic form induced by the above bilinear form.
As $\Omega(\hat Q)$ is irreducible, we obtain that $K_{11}$ 
is irreducible on $U_{11}$. 
Replacing $T_1$ by $T_2$, the same argument shows
that $K_{11}$ is irreducible on $U_{11}^\perp$. 
\end{proof}

The proof of Theorem~\ref{4gen} is now straightforward.

\begin{proof}[The proof of Theorem~$\ref{4gen}$]
As explained in Lemma~\ref{goodU}, $G_i$ is UCS if and 
only if $\GL(V)_{U_i}$ is irreducible on both $V$ and $\Lambda^2V/U_i$.
Hence Lemmas~\ref{notucs} and~\ref{goodU} imply Theorem~\ref{4gen}.
\end{proof}

UCS $p$-groups with exponent $p^2$ are studied in the next section. 
It follows from
Theorems~\ref{4gen} and \ref{3th} that a 4-generator
exponent-$p^2$ UCS $p$-group is a quotient of either $H_{p,4}/N_{16}$ or
$H_{p,4}/N_{18}$. 
Determining precisely which of these two cases leads to UCS $p$-groups 
involves subtle
isomorphism problems which depend on the value of the prime~$p$.
This is illustrated in the the next section.

\section{Exterior self-quotient modules}\label{s7}

The study of UCS $p$-groups with exponent $p^2$ is reduced by
Theorem~\ref{UCSconstr}(b) to considering a problem in
representation theory. 
Recall that the concepts of ESQ-modules and ESQ-groups were defined 
in Section~\ref{ucsexist}.
Unlike the property of irreducibility, the ESQ-property is preserved
under subgroups and field extensions, as shown by the following lemma.

\begin{lemma}\label{ESQstructure}
Let $V$ be an ESQ $\,\F G$-module. Then
\begin{itemize}
\item[(a)] every subgroup $H$ of $G$ is also an ESQ-subgroup of $\GL(V)$.
\item[(b)] $V\otimes_{\F}\E$ is an ESQ
$\E G$-module for every extension field $\E$ of $\F$.
\item[(c)] $G$ contains no non-trivial scalar matrices, and $\dim(V)\geq 3$.
\end{itemize}
\end{lemma}

\begin{proof}
Parts (a) and (b) are routine to verify. 
If a scalar matrix $\lambda I$
lies in $G$, then it follows from $\VwedgeV/U\cong V$ that
$\lambda^2=\lambda$ and hence $\lambda=1$. In addition,
$\dim(\VwedgeV)\geq\dim(V)$ implies $\dim(V)\geq3$. Hence part (c) holds.
\end{proof}

An irreducible ESQ-module can give rise to a smaller
dimensional irreducible module over a larger field which does not enjoy the
ESQ-property. For example, a 5-dimensional irreducible, but not
absolutely irreducible, ESQ-module over $\F_q$, gives rise to a
one dimensional irreducible module over $\F_{q^5}$ which is not an
ESQ-module by Lemma~\ref{ESQstructure}(c).

Let $r$ and $q$ be coprime integers. Denote the order of $q$ modulo $r$ by
$\ord_r(q)$. Then $\ord_r(q)$ is the
smallest positive integer $n$ satisfying $q^n\equiv 1\pmod r$.
\textsc{Warning:}~The variables $p$, $r$, and $G$ have different meanings
in the following discussion about ESQ-groups, to the previous discussion about
UCS-groups. 

\begin{theorem}\label{d=p}
Let $p$ be a prime, $q$ a power of a prime (possibly distinct from $p$), 
and let $G$ be a minimal irreducible ESQ-subgroup
of $\GL_p(q)$. Then one of the following holds: 
\begin{itemize}
\item[(a)] $G$ is not absolutely irreducible, $r:=|G|$ is prime,
$\ord_r(q)=p$ and there exist {\it distinct}
$\alpha,\ \beta\in\langle q\rangle\leq\F_r^\times$ such that $\alpha+\beta=1$;
\item[(b)] $G$ is an absolutely irreducible non-abelian simple group;
\item[(c)] $G$ is absolutely irreducible, $|G|=pr^s$ where $r$ is a
prime different to $p$, $\ord_r(q)=1$, and $s=\ord_p(r)$. Moreover,
$G'$ is an elementary abelian group of order $r^s$ and
$G/G'$ has order $p$ and acts irreducibly on $G'$.
\end{itemize} 
\end{theorem}

\begin{proof}
Denote by $V=(\F_q)^p$ the corresponding ESQ $\F_q G$-module. Here
$p$ and $\textup{char}(\F_q)$ may be distinct primes.
By Lemma~\ref{ESQstructure}(a), subgroups of ESQ-groups are ESQ-groups,
and hence by the minimality of $G$, proper subgroups of $G$ act
reducibly on $V$. If $H$ is a 
non-trivial abelian normal subgroup of $G$, then by Clifford's theorem
either $H$ acts irreducibly on $V$, or
$V=V_0\oplus V_1\oplus\cdots\oplus V_{p-1}$ is an internal
direct sum of $p$ pairwise non-isomorphic 1-dimensional
$H$-submodules. (If $V_0,V_1,\dots,V_{p-1}$ were all
isomorphic, then $H$ would contain non-trivial scalar matrices contrary to Lemma~\ref{ESQstructure}(c).)

(a) Suppose that $G$ does not act absolutely irreducibly on $V$.
By Lemma~\ref{ESQstructure}(b) we may view $G$ as an ESQ-subgroup of
$\GL_p(\E)$ where $\E$ denotes the algebraic closure of $\F_q$. The
module $\E^p$ is a direct sum of $p$ pairwise non-isomorphic
algebraically conjugate irreducible 1-dimensional $G$-submodules
by \cite[Theorem~VII.1.16]{HB}. This proves that $G$ is abelian. We argue
that $r:=|G|$ 
is prime. If not, then $G$ has a proper non-trivial subgroup $H$. By
the first paragraph of the proof, $V$ decomposes as the sum of 
$1$-dimensional $H$-modules. In particular, an element 
$h\in H$ has an eigenvalue, 
$\lambda$, say, in $V$. Since $h$
commutes with $G$, the linear transformation $h$ is a $G$-endomorphism
of $V$, and so are the transformations $\lambda I$ and $h-\lambda I$. 
As $h-\lambda I$ is not invertible, Schur's lemma shows that $h-\lambda I=0$, 
and so $h$ coincides with the scalar matrix $\lambda I$.
However, as, by Lemma~\ref{ESQstructure}(c), $G$ contains no non-trivial scalar matrices,  we obtain that $h=I$. 
Hence the only proper 
subgroup of $G$ is the trivial subgroup,  which shows that $r$ is prime.

Then $\E^p=W\oplus\sigma(W)\oplus\cdots\oplus\sigma^{p-1}(W)$
where $\sigma\colon\E\to\E$ is the $q$th power (Frobenius)
automorphism, and $\sigma(W)$ denotes an $\E G$-module algebraically
conjugate to $W$ via $\sigma$ \cite[Definition~VII.1.13]{HB}.
Further, $\sigma^p(W)\cong W$. The exterior
square of $\E^p$ is isomorphic to a direct sum
$\sum_{0\leq i<j<p} \sigma^i(W)\wedge\sigma^j(W)$. 
Thus $W\cong\sigma^i(W)\wedge\sigma^j(W)$ for some $0\leq i<j<p$, as
$\E^p$ is an ESQ-module. Suppose that $G=\langle
g\rangle$. Then $g$ is conjugate in $\GL_p(\E)$ to a diagonal matrix
$\diag(\zeta,\zeta^q, \dots,\zeta^{q^{p-1}})$ where
$\zeta\in\E^\times$ has order $r$. It follows from $q^p\equiv 1\pmod
r$, that $\ord_r(q)$ equals 1 or $p$. The first possibility does not
arise as $g$ is not a scalar matrix. The
condition $W\cong\zeta^i(W)\wedge\zeta^j(W)$ implies that
$1\equiv q^i+q^j\pmod r$ where $\alpha=q^i$ and $\beta=q^j$ are distinct
powers of $q$. This completes the proof of part (a).

(b) Suppose now that $G$ is a simple group acting absolutely irreducibly
on $V$. Then $G$ must be non-abelian.

(c) Suppose now that $G$ acts absolutely irreducibly on $V$ and $G$ is
not simple. Let $N$ be minimal normal subgroup of $G$.
Then $N$ is a proper subgroup of $G$ and so acts reducibly.
Let $V=V_0\oplus\cdots\oplus V_{p-1}$ be a direct sum
of $p$ irreducible 1-dimensional $N$-submodules. 
It follows that $N$ is abelian. Suppose that $|N|=r^s$ where $r$ is prime.
By the first
paragraph of this proof, $V_0,\dots,V_{p-1}$ are pairwise non-isomorphic.
Since $N$ has a non-trivial 1-dimensional module over $\F_q$, it
follows that $r$ divides $q-1$, or $\ord_r(q)=1$.

By Clifford's theorem, $G$ acts transitively on the set
$\{V_0,\dots,V_{p-1}\}$. Choose $g\in G$ that induces a $p$-cycle. By
renumbering if necessary, assume that $V_i g=V_{i+1}$ where the
subscripts are read modulo $p$. Choose $0\ne e_0\in V_0$ and set
$e_i=e_0g^i$. Then $e_0g^p=\lambda e_0$ for some
$\lambda\in\F_q$. Since $g^p$ is the scalar matrix $\lambda I$, it
follows from Lemma~\ref{ESQstructure}(c) that $\lambda=1$. Since
$\langle g\rangle N$ acts irreducibly on $V$, it follows by
minimality that $G=\langle g\rangle N$ has order $pr^s$.

We may view $G$ as a subgroup of the wreath product $C_r\Wr C_p$. The
base group $(C_r)^p$ may be identified with the vector space
$(\F_r)^p$, and $N$ may be identified with an irreducible
$\F_r C_p$-submodule of $(\F_r)^p$. Since the derived subgroup
$G'$ equals $N$, it follows that $r\ne p$. By
Maschke's theorem $(\F_r)^p$ is a completely reducible
$\F_r C_p$-module. Since no $V_i$ is the trivial module, $N$
corresponds to an irreducible $\F_r C_p$-submodule of $(\F_r)^{p-1}$.
However, $(\F_r)^{p-1}$ is the direct sum of $(p-1)/\ord_p(r)$
irreducible $\F_r C_p$-submodules each of dimension $\ord_p(r)$.
This proves that $s=\ord_p(r)$, and completes the proof.
\end{proof}

In Theorem~\ref{d=p}(a,c), the order $|N|=r^s$ of a minimal normal
subgroup of $G$ is severely restricted. If $p$ and $q$ are given, then
$r^s$ must divide $|\GL_p(\F_q)|$, however, if $p$ is given and $q$
is arbitrary, then there are still finitely many choices for $r^s$. 

\begin{theorem}\label{d=pESQ}
Let $G\leq\GL_p(q)$ be as in Theorem~$\ref{d=p}$\textup{(a,c)}, and let $N$ be
a minimal normal subgroup of $G$. Let $j\in\{2,3,\dots,p-1\}$, and let
$Y_{p,j}$ be the $\Z$-module
\[
  Y_{p,j}=\langle e_0,e_1,\dots,e_{p-1}\mid
    e_k=e_{1+k}+e_{j+k},\quad k=0,1,\dots,p-1\rangle
\]
where the subscripts are read modulo $p$. Then 
the $\Z$-modules $Y_{p,j}$ are finite, and $N$ is a factor group
of $Y_{p,j}$ for some $j$. 
\end{theorem}

\begin{proof}
Cases (a) and (c) can be unified by considering the potentially larger
finite field $\E=\F_q(\zeta)$ where $\zeta$ has order $r$. In case (a),
the group $G=N$ has order prime $r$, and 
$
\E^p=V_0\oplus
V_1\oplus\cdots\oplus V_{p-1}
$ 
where $V_k$ is an irreducible
1-dimensional $\E N$-module corresponding to the eigenvalue 
$\sigma^k(\zeta)$ where $\sigma\in\Gal(\E/\F_q)$ has order~$p$. Let
$g\in\GL_p(\E)$ be a matrix of order $p$ satisfying
$V_k=V_0g^k$, for all $k$. Then $\langle g\rangle N$ is a minimal irreducible
ESQ-subgroup of $\GL_p(\E)$ with $s=1$.

Thus we reduce to case (c) where $G=\langle g\rangle N$ is a minimal
irreducible ESQ-subgroup of $\GL_p(\E)$ where $g$ is as above,
$|N|=r^s$ is elementary abelian and $\E^p=V_0\oplus
V_1\oplus\cdots\oplus V_{p-1}$ is a sum of irreducible (but not
necessarily algebraically conjugate) 1-dimensional $\E N$-submodules.
As $V_k=V_0g^k$, each $V_k$ is a non-trivial $N$-module.
If $V_0\cong V_0\wedge V_j$, then $V_j$ would be the
trivial $N$-module. Hence we have $V_0\cong V_i\w V_j$ where $0<i<j<p$.

By replacing $g$ by $g^i$, we may assume that $V_0\cong V_1\w V_j$, that is,
we may assume that $i=1$.
Suppose henceforth that $V_0\cong V_1\wedge V_j$ as $N$-modules
where $1<j<p$ and $V_k=V_0g^k$. Suppose $n\in N$ and
$n=\diag(\zeta^{x_0(n)},\zeta^{x_1(n)},\dots,\zeta^{x_{p-1}(n)})$
where $x_k(n)\in\F_r$ and $\zeta\in\E$ has order~$r$. The $N$-isomorphisms
$V_k\cong V_{1+k}\wedge V_{j+k}$ give rise to equations
$x_k(n)=x_{1+k}(n)+x_{j+k}(n)$ in $\F_r$, where the subscripts are
read modulo $p$. We shall view $x_k$ as an element of the
dual space $N^*$ of $N$. Thus $x_k=x_{1+k}+x_{j+k}$ in $N^*$ for all~$k$.

Our goal now is to relate the abelian group $Y_{p,j}$, which is
defined in terms of $p$ and~$j$, to the abelian group $N$ of order $r^s$.
We show now that $N^*$ is an epimorphic image
of $Y_{p,j}$. Let $\Z^p=\langle e_0,\dots,e_{p-1}\rangle$
where $e_k=(0,\dots,0,1,0,\dots,0)$ has a~1 in position~$k+1$.
First, note that $N^*=\langle x_0,x_1,\dots,x_{p-1}\rangle$
because $\cap_{k=0}^{p-1} \ker(x_k)$ is trivial. Second, the homomorphism
$\Z^p\to N^*$ defined by $\sum_{k=0}^{p-1} y_ke_k\mapsto
\sum_{k=0}^{p-1} y_kx_k$ is surjective by the previous sentence, 
and its kernel contains
$R_j=\langle e_k-e_{1+k}-e_{j+k}\mid k=0,1,\dots,p-1\rangle$
since $x_k-x_{1+k}-x_{j+k}=0$ in $N^*$ for all $k$. As $Y_{p,j}\cong\Z^p/R_j$,
this homomorphism induces an epimorphism $Y_{p,j}\to N^*$.
We prove below that $Y_{p,j}$ is finite.
Hence $|N^*|=|N|=r^s$ divides~$|Y_{p,j}|$. 

The above subgroup $R_j$ of $\Z^p$ equals $\Z^p(I-C-C^j)$ where
$C=C_p$ denotes the cyclic permutation matrix
\[
  C_p=\begin{pmatrix}0&1& &0\\ & &\ddots& \\0&0& &1\\1&0&&0\end{pmatrix}
  \in\GL_p(\Z).
\] 
We argue that
$|\Z^p:R_j|$ is finite, or equivalently that $\det(I-C-C^j)\ne0$. 
Denote by $\,\overline{\phantom{w}}\,$ the matrix ring homomorphism
$\,\overline{\phantom{w}}\colon\Z^{\,p\times p}\to\F_p^{\,p\times p}$.
Then $\overline{C}$ is conjugate in $\GL_p(\F_p)$ to an
upper-triangular matrix with all eigenvalues $1$, and
$\overline{I-C-C^j}$ is conjugate to an upper-triangular matrix with
all eigenvalues $-1$. Therefore
\[
  \det(I-C-C^j)\equiv (-1)^p \pmod p.
\]
Thus $\det(I-C-C^j)\ne0$ and $|Y_{p,j}|=|\det(I-C_p-C_p^j)|$ is finite.
Alternatively, the formula for the determinant of a
circulant matrix shows that
\[
  \det(I-C_p-C_p^j)=\prod_{k=0}^{p-1}(1-\xi_p^k-\xi_p^{jk})\ne 0
\]
for $1<j<p$, where $\xi_p=e^{2\pi i/p}$ denotes a complex primitive
$p$th root of 1.
\end{proof}

The previous determinant can be defined for any size $n$. 
Let $\delta_{n,j} = \det(I-C_n-C_n^j)$ for $1<j<n$.
However, for some composite $n$, this determinant can be zero.
For example, if $n\equiv 0\pmod 6$ and $j\equiv -1\pmod 6$,
then $\delta_{n,j} = 0$ as $1-\xi_n^{n/6}-\xi_n^{-n/6}= 0$. We can observe that
$\delta_{n,n-1} =\pm1$ whenever $n\equiv\pm 1\pmod 6$.
To see this note that
$(I-C_n-C_n^{-1})^{-1}=I+\sum_{i=1}^{(n-2)/3}(-C_n)^{3i}(I+C_n^{n-1})$
when $n\equiv-1\pmod 6$. A similar formula holds for the inverse
when $n\equiv1\pmod 6$. Furthermore, one can prove using row operations, and
basic properties of Fibonacci numbers that
$\delta_{n,2}=1+(-1)^n-F_{n-1}-F_{n+1}$.
Assume henceforth that $n=p$ is prime.

The structure of the abelian group $Y_{p,j}$ can be determined from the
elementary divisors of the Smith normal form of the matrix $I-C_p-C_p^j$. 
The following table shows, for example, that
$Y_{7,3}\cong Y_{7,5}\cong (C_2)^3$ has order 8 and exponent 2, and
$Y_{13,12}$ is trivial.

\vskip3mm
\def\vstr{\phantom{\kern-20pt$\sum_{i_j}^{k^l}$}}
\def\k{\kern-1.3pt}
\begin{center}
\begin{tabular}{|c|c|c|c|c|c|c|c|c|c|c|c|c|c|c|c|c|c|}\hline
$p$&3&5&\k5\k&7&7&\k7\k&11&11&11&\k11\k&13&13&13&13&13&\k13\k\\ \hline
$j$&2&2,3&\k4\k&2,4&3,5&\k6\k&2,6&3,4&5,7,8,9&\k10\k&2,7&3,9&4,10&5,8&6,11&\k12\k\\ \hline
\text{Divisors}\vstr&$2^2$&11& &29&$2^3$& &199&67&23& &521&131&79&$3^3$&53&\\ \hline
\end{tabular}
\nopagebreak[4]\vskip2mm
Table 1: The elementary divisors of $Y_{p,j}$ that are not 1.
\end{center}

Table~1 suggests that $Y_{p,j}\cong Y_{p,k}$ when $jk\equiv1\pmod p$.
Note that there exists a
permutation matrix which conjugates $C_p$ to $C_p^k$, and this conjugates
$I-C_p-C_p^j$ to $I-C_p^k-C_p$. Hence $Y_{p,j}\cong Y_{p,k}$.

The following theorem shows that case (b) of Theorem~\ref{d=p}
does not arise in dimension~$5$. 

\begin{theorem}\label{d=5}
Let $q$ be a prime power, and let $G$ be a minimal irreducible
ESQ-subgroup of $\GL_5(q)$. Then case~\textup{(b)} of Theorem~$\ref{d=p}$
does not arise, and more can be said about cases~\textup{(a)}
and ~\textup{(c):}
\begin{itemize}
\item[(a)] $G$ is not absolutely irreducible, $|G|=11$, and $\ord_{11}(q)=5$,
\item[(c)] $G$ is absolutely irreducible of order $55$ and $\ord_{11}(q)=1$.
\end{itemize} 
Furthermore, both of these possibilities occur.
\end{theorem}

\begin{proof}
Suppose that $G$ satisfies case (a) of Theorem~\ref{d=p}. Then $r=11$
by Theorem~\ref{d=pESQ} and Table~1. The subgroup $\langle
q\rangle\leq\F_{11}^\times$ has order $p=5$. Therefore
$\langle q\rangle=\{1,3,4,5,9\}=(\F_{11}^\times)^2$, 
and $\ord_{11}(q)=5$ implies that
$q\equiv 3,4,5,9\pmod{11}$. Note that $\alpha=3$
and $\beta=9$ satisfy $\alpha+\beta=1$ in $\F_{11}$. Conversely, if
$\ord_{11}(q)=5$, then the cyclotomic polynomial
$\Phi_{11}(x)=x^{10}+\cdots+x+1$ factors over $\F_q$ as a product of
two distinct irreducible quintics. The companion matrix of either of these
quintics generates an irreducible (but not absolutely irreducible)
ESQ-subgroup of $\GL_5(\F_q)$ of order $11$. 

Suppose now that $G$ satisfies case~(c) of Theorem~\ref{d=p}. As
above, $r=11$. Furthermore, $\ord_{11}(q)=1$ and
$s=\ord_5(11)=1$. Therefore $G$ is isomorphic to 
the group $\langle g,n\mid g^5=n^{11}=1,g^{-1}ng=n^t\rangle$ of order
55, where $\ord_{11}(t)=5$. By replacing $g$ by a power of itself, we
may assume that $t=3$.
Conversely, there is an irreducible ESQ-subgroup $G$ of $\GL_5(\F_q)$ of
order $55$ if $q^5\equiv1\pmod{11}$. To see this apply
Theorem~\ref{AGammaL} below with $G=G_L$ where
$L=(\F_{11}^\times)^2=\{1,3,4,5,9\}$. Note that
$\alpha=3, \beta=9$ satisfy $\alpha+\beta=1$ in $\F_{11}$, and the
condition $q\in L$ is equivalent to $q^5\equiv1\pmod{11}$. 
Here $G_L$ is a {\em minimal} ESQ-subgroup if an only if
$q\equiv1\pmod{11}$. 

Suppose now that case~(b) of Theorem~\ref{d=p} holds, and $G$
is a non-abelian simple (absolutely) irreducible ESQ-subgroup
of $\GL_5(q)$ where
$q=p^k$ and $\textup{char}(\F_q)=p$. As $G$ is non-abelian simple, we have
$G\leq\SL_5(q)$ and $G\cap Z(\SL_5(q))=1$. Thus $G$ is isomorphic to
an irreducible subgroup of $\PSL_5(q)$. Consider first the case when $p=2$.
The irreducible subgroups of $\PSL_5(2^k)$ were classified by
Wagner~\cite{Wag78}. As $G$ is simple, it must be isomorphic to one of the
following groups: $\PSL_2(11)$, $\PSL_5(2^\ell)$ where $\ell\,|\, k$, or
$\PSU_5(4^\ell)$ where $2\ell\,|\, k$. Furthermore, each of these possibilities
give irreducible subgroups of $\PSL_5(2^k)$. We shall use
Lemma~\ref{ESQstructure}(a) to show that none of these possibilities
give irreducible ESQ-subgroups of $\GL_5(q)$. Since
$\PSL_5(2)\leq\PSL_5(2^\ell)$ and $\PSL_2(11)\leq\PSU_5(4^\ell)$ for
all~$\ell$, minimality implies that $G$ is an irreducible (and hence
absolutely irreducible) group isomorphic to
$\PSL_2(11)$ or $\PSL_5(2)$. The Atlas~\cite{Atlas} lists (up to isomorphism)
the irreducible 5-dimensional modules for $\PSL_2(11)$ and $\PSL_5(2)$
in characteristic~2. There are four: two for $\PSL_2(11)$ over $\F_4$, and
two for $\PSL_5(2)$ over $\F_2$. Straightforward computation shows that
the exterior square of each of these is irreducible. Thus no ESQ-groups
arise when $p=2$.

Suppose now that $p>2$. Here we use the classification of irreducible
subgroups of $\PSL_5(p^k)$ in~\cite{DW79}. To prove that there are no
non-abelian simple irreducible ESQ-subgroups of $\PSL_5(p^k)$ (really
$\GL_5(p^k)$), it will be convenient by Lemma~\ref{ESQstructure}(b)
to choose $q=p^k$ to be ``sufficiently large.'' It follows from~\cite{DW79}
that $G$ is isomorphic to one of the following: $\PSL_5(p^\ell)$,
$\PSU_5(p^{2\ell})$, $\POmega_5(p^\ell)$, $\POmega_5(3)$, $\PSL_2(p^\ell)$,
$A_5$, $A_6$, $\PSL_2(11)$, $A_7$, $\textup{M}_{11}$. Each of these 
groups contains
a subgroup isomorphic to the alternating group $A_4$. Indeed,
$\PSL_2(p^\ell)$ contains such a subgroup by~\cite[Satz~II.8.18]{huppert}, 
and we also have
\[
  A_5\leq\PSL_2(11)\leq \textup{M}_{11}, A_5\leq\POmega_5(p)\leq\PSL_5(p^\ell),
  \POmega_5(p)\leq\PSU_5(p^{2\ell}), \textup{and } A_5\leq\POmega_5(3).
\]
As $A_4$ is a subgroup of $A_5$, the claim is valid.

Next we show that $A_4$ does not have a $5$-dimensional faithful ESQ-module. 
Suppose that $V$ is a faithful $5$-dimensional $A_4$-module over
$\F_q$. Suppose first that the characteristic of $\F_q$ is at least $5$, and
so the $A_4$-modules are completely reducible. If
$q$ is large enough, which we may assume by Lemma~\ref{ESQstructure}(b), 
then there are 4 pairwise non-isomorphic irreducible $A_4$-modules
over $\F_q$, three of which are 1-dimensional, and one is 3-dimensional. Since
$A_4$ is non-abelian, $V$ decomposes as $V=V_3+V_1+V'_1$ where the subscript
denotes the dimension. Thus $\Lambda^2 V=\Lambda^2(V_3+V_1+V'_1)$ contains
three 3-dimensional and a 1-dimensional direct summand, and so it is not
ESQ. If the characteristic of $\F_q$ is $3$, then there are only two
irreducible $A_4$-modules, one is 1-dimensional, and the other is
$3$-dimensional. 
Thus we may argue the same way as above, except we must use composition
factors instead of direct summands.
\end{proof}

Next we determine the 
ESQ-subgroups in $\GL_4(\F)$ for fields $\F$ of characteristic
different 
from 2. The characteristic 2 case would require additional
considerations and it is not relevant to
Theorem~\ref{main}(d). Note that apart from $\mathrm{char}(\F)\ne
2$ we allow $\F$ to be arbitrary,
not necessarily a prime field, and it can also have characteristic zero.

Let $L=\AGL_1(5)$ be the group of linear functions of the
5-element field considered as a permutation group of degree 5. 
We have $L=\langle a, b \rangle$, where $a=(0\,1\,2\,3\,4)$,
$b=(1\,2\,4\,3)$, and $|L|=20$. Then
$L$ naturally embeds into $\GL_5(\F)$ for any field $\F$. If the
characteristic of $\F$ is different from 5, then the underlying
module splits into a direct sum of submodules $\F^5=V\oplus V_1$, where
$V=\{(x_0,\dots,x_4) \mid x_0+\dots+x_4=0\}$ and
$V_1=\{(x,\dots,x)\mid x\in \F\}$. The action of $L$ on $V$ is
absolutely irreducible and it is the only faithful irreducible
representation of $L$ over $\F$. In the following theorem 
$L$ and $V$ will denote what have just been defined.

\begin{theorem} \label{d=4}
Let $\F$ be a field of characteristic different from $2$ and let 
$K\leq\GL_4(\F)$ be a finite irreducible ESQ-subgroup. Then 
$\mathrm{char}(\F)\ne5$, $K$ is
isomorphic to a subgroup of $L$ and the action of $K$ is isomorphic to
the restriction of the action of $L$ on $V$. Moreover, $5$ divides the
order of $K$, and if $5$ is a square in $\F$ then $K\cong L$.
\end{theorem}

\begin{proof}
First recall \cite[Prop.~5.5.10]{kl} that no finite non-abelian simple group has a
non-trivial representation of degree two over a field of characteristic
different from 2.

Let $M$ be a minimal normal subgroup of $K$. We first show that $M$ is
abelian. If not, then $M$ is the direct product of pairwise isomorphic
non-abelian simple groups. Let $S$ be one of the simple
factors. Applying Clifford's Theorem twice for 
$S \vartriangleleft M \vartriangleleft K$,
and considering that $S$
has no two-dimensional non-trivial representation, we conclude that $S$
is irreducible. 
Let $V=(\F_q)^4$. Since $(\VwedgeV)/U \cong V$ for some
2-dimensional $S$-submodule $U$, the remark in the first paragraph in this
proof gives that  $S$  acts trivially on  $U$.  

Let $V^*$ denote the dual space of $V$. Let 
$\psi:\Hom(V,V^*)\rightarrow V\otimes V$ be  defined as follows.
Let $x_1,\ldots,x_4$ be a basis of $V$ and let $x_1^*,\ldots,x_4^*$ be the
dual basis of $V^*$. If $f\in\Hom(V,V^*)$ represented 
by the matrix $(\alpha_{i,j})$ with respect to these bases,
then let $\psi(f)$ be the element
$\sum_{i,j}\alpha_{i,j}x_i\otimes x_j$. It is easy to check that
$\psi$ is a linear isomorphism. Now the group $S$ acts on both 
spaces $\Hom(V,V^*)$ and $V\otimes V$: if $g\in S$, then the matrix of
$f^g$ is $g^\T (\alpha_{i,j})g$. An easy calculation shows that the isomorphism $\psi$ is
an isomorphism of $S$-modules, and so the fixed points of $S$ in $V\otimes V$ 
correspond to
intertwining operators between the $S$-modules $V$ and $V^*$.
Identify $\Lambda^2 V$ with
$\langle u\otimes v-v\otimes u\mid u,v\in V\rangle\subseteq V\otimes V$.
As $U$ is a 2-dimensional
subspace of $\Lambda^2 V$ on which $S$ acts trivially,
the dimension
of these intertwining operators is at least~2.
The dimension of the space of these intertwining
operators is equal to the dimension of the centralizer algebra of 
the $S$-module $V$. On the other hand, by Schur's lemma, the centralizer 
algebra is a quadratic extension field $\E$ of $\F_q$. 
Further, the $S$-module $V$ is also an $\E S$-module.
This means that 
$S$ can be viewed as a subgroup of $\GL_2(\E)$, which contradicts the first
paragraph of the proof.

So we have that $M$ is an elementary abelian $r$-group for some prime
$r$, which cannot be the characteristic of $\F$. Take an extension
field $\E \supseteq \F$ containing primitive $r$th roots of unity and
consider $M\leq\GL_4(\E)$, which, by Lemma~\ref{ESQstructure}(a,b), 
is an ESQ-group. Now $\E$ is a splitting field for $M$, so
we can fix 
an eigenbasis $e_1, e_2, e_3, e_4 \in \E^4$ of $M$.

Suppose that $M$ contains an element without fixed points, i.e., an
element $g\in M$ such that 1 is not an eigenvalue of $g$. Let the
eigenvalues of $g$ be $\lambda_i\in \E$ ($1\leq i\leq 4$). Then the
eigenvalues of $g \wedge g$ are $\lambda_j \lambda_k$ ($1\leq j<k\leq
4$). By the ESQ property there is an injective map 
$i\mapsto P(i)=\{j,k\}$
such that $\lambda_i=\lambda_j\lambda_k$. Since 1 is not among the
eigenvalues of $g$, we see that $i\notin P(i)$. Up to renumbering the
eigenvalues there are only two essentially different injective maps
satisfying this property. So we arrive at two alternative 
systems of equations:
\begin{equation}\label{eqs}
\lambda_1=\lambda_2 \lambda_3, \, \lambda_2=\lambda_3 \lambda_4, \,
\lambda_3=\lambda_4 \lambda_1, \, \lambda_4=\lambda_1 \lambda_2;
\end{equation}
and
$$
\lambda_1=\lambda_2 \lambda_3, \, \lambda_2=\lambda_1 \lambda_4, \,
\lambda_3=\lambda_1 \lambda_2, \, \lambda_4=\lambda_1 \lambda_3.
$$
It is easy to solve these systems of equations. In the first case we
obtain
$$
\lambda_1=\epsilon, \, \lambda_2=\epsilon^2, \,
\lambda_3=\epsilon^4, \, \lambda_4=\epsilon^3, 
$$
where $\epsilon^5=1$, and that implies $r=5$. 
In the second case the solutions have the form
$$
\lambda_1=\epsilon^2, \, \lambda_2=\epsilon^3, \,
\lambda_3=\epsilon^5, \, \lambda_4=\epsilon,
$$
where $\epsilon^6=1$. However, non-trivial elements of $M$ have prime
order $r$, hence either $\lambda_1=\epsilon^2=1$ or
$\lambda_2=\epsilon^3=1$, contrary to our assumption that 1 is not an
eigenvalue of $g$. So the only possibility is that such an element has
order 5 and its eigenvalues are all the four distinct primitive fifth
roots of unity.

Next we consider the case when 1 is an eigenvalue of every element
$g\in M$. We are going to show that this is not possible. 
Since $K$ is irreducible and $M\vartriangleleft K$, we have
$\{v\in \F^4\mid \forall g\in M:\, vg=v\}=0$. 
This implies that the trivial module is not a
direct summand in the $M$-module $\F^4$, and so it cannot be a direct
summand in $\E^4$. Hence 
$\{v\in \E^4\mid \forall g\in M:\, vg=v\}=0$. 
Now for every fixed $i$ ($1\leq i\leq 4$) the number of elements 
$g \in M$ with $e_ig=e_i$ equals $|M|/r$. 
If an element $g\in M$ has a fixed point in $\E^4$
(that is, 1 is an eigenvalue of $g$) then $g$ must fix one of the basis
vectors $e_i$. This shows, for $r>4$, that the number of elements in $M$
without eigenvalue $1$ is at least $|M|-4|M|/r>0$, which is not the case
now. Hence $r=2$ or $r=3$. If $M$ were cyclic, then all non-trivial
elements of $M$
would have non-trivial eigenvalues, hence $M$ is non-cyclic in
our present case. Then $M$ intersects $\SL_4(\F)$ non-trivially,
so by the minimality of $M$ we have that all matrices in $M$ have
determinant 1. 

Let $r=2$. 
Let $D$ denote the 8-element subgroup consisting
of diagonal matrices (with respect to the basis $e_1, e_2, e_3, e_4$)
with diagonal entries $\pm1$ and with determinant
1. We have $M\leq D$. By Lemma~\ref{ESQstructure}(c),
$M$ cannot contain $-I$. Now $D$ has seven maximal subgroups, 
out of these three contain $-I$ and the remaining
four are the stabilizers of the four basis
vectors $e_i$ ($1\leq i\leq 4$) in $D$. So 
there remains no possibility for $M$.

Let $r=3$, and denote by $\omega\in \E$ a primitive third root of
unity. Let $g\in M\leq\SL_4(\E)$ and suppose that $g\neq 1$.
Then 1 is an eigenvalue of $g$. If the
multiplicity of the eigenvalue 1 is one, then, as $\det g=1$, 
the eigenvalues of $g$
are $1,\omega,\omega,\omega$ or $1,\omega^2,\omega^2,\omega^2$. In
both cases 1 is not an eigenvalue of $g \wedge g$ contradicting the
ESQ property. Hence for every $1\ne g\in M$ the 
multiplicity of the eigenvalue 1 is at least two, and then 
we infer that the 
eigenvalues of $g$ are $1, 1, \omega, \omega^2$. 
Now $M$ is a proper subgroup of the group of
diagonal matrices with order $3^3$ and determinant 1, so it is
generated by at most two elements. Since no basis vector 
can be fixed by both generators, each 
one of the basis vectors $e_1, e_2, e_3, e_4$ is fixed by one generator, and the eigenvalue for the other
generator on the same eigenvector must be $\omega$ or $\omega^2$.
Then the product of the two
generators does not have eigenvalue 1, contrary to our assumption.

In summary, we have proved that $r=5$ and there is a $g\in M$ which
has a diagonal matrix $\mathrm{diag}(\epsilon, \epsilon^2, \epsilon^4,
\epsilon^3)$ with respect to the basis $e_1,e_2,e_3,e_4\in \E^4$
(where $\epsilon\in \E$ is a fifth root of unity). Now $g \wedge g$ is 
$\mathrm{diag}(\epsilon, \epsilon^2, \epsilon^4,
\epsilon^3, 1, 1)$ with respect to the basis $ e_2 \wedge e_3, e_3
\wedge e_4, e_4 \wedge e_1, e_1 \wedge e_2, e_1 \wedge e_3, e_2 \wedge
e_4$ of $\Lambda^2 \E^4$. Hence any isomorphism from the $\langle g
\rangle$-module $\E^4$ into $\Lambda^2 \E^4$ maps $e_i$ to a multiple of
$e_{i+1} \wedge e_{i+2}$ (where the indices are taken
modulo 4). 

Now take an element $h\in\mathbf{C}_K(g)\leq\GL_4(\E)$. Since
the eigenvalues of $g$ are distinct, $h$ is also diagonal
with respect to the basis $e_1, e_2, e_3, e_4$, say,
$h=\mathrm{diag}(\lambda_1, \lambda_2, \lambda_3, \lambda_4)$. The
matrix of $h \wedge h$ restricted to
$\langle e_2 \wedge e_3 ,e_3 \wedge e_4 ,e_4
\wedge e_1 ,e_1 \wedge e_2 \rangle$ is 
$\mathrm{diag}(\lambda_2\lambda_3, \lambda_3\lambda_4, \lambda_4\lambda_1, 
\lambda_1\lambda_2)$. From the ESQ property it follows that the
eigenvalues of $h$ satisfy the same system of equations~\eqref{eqs} as above,
therefore $h$ is a power of $g$. 
Thus we have shown that $\langle g \rangle$ is a
self-centralizing subgroup of $K$. In particular, we obtain that 
$M=\langle g \rangle$ is cyclic of order~5. 

Since $M\vartriangleleft K$ and $M$ is self-centralizing, $K$ is
isomorphic to a subgroup in the holomorph of $M$, which is $L$. Since
the characteristic polynomial of any element $g$ generating $M$ is
$x^4+x^3+x^2+x+1$, $M\leq\GL_4(\F)$ is unique up to
conjugacy. Hence $K$ can be embedded into a subgroup of
$\GL_4(\F)$ that is isomorphic to $L$.

There are three subgroups of $L$ containing $M$; namely, $M$, a dihedral group
of order 10, and $L$. If 5 has a square root in $\F$, then the
dimensions of the irreducible representations of the dihedral group
$D_5$ are at most two, hence proper subgroups of $L$ cannot act
irreducibly on $\F^4$ in this case.
\end{proof}

Theorem~\ref{main}(d) is an immediate consequence of the following results.

\begin{theorem} \label{thm-expsquare}
Let $p$ be an odd prime, and let $\{x_1,x_2,x_3,x_4\}$ be a generating
set for $H_{p,4}$.
\begin{itemize}
\item[(a)] There is no $4$-generator UCS\, $5$-group with exponent $5^2$. 
\item[(b)] If $p \equiv \pm 1 \pmod{5}$, then there is a unique isomorphism
class of $4$-generator UCS $p$-groups with exponent $p^2$, namely,
\begin{equation*}
\begin{aligned}
G_1=H_{p,4}/\langle 
&x_1^p [x_1,x_3][x_4,x_1][x_2,x_3]^2[x_4,x_2]^2[x_4,x_3]^2,\\
&x_2^p [x_2,x_1][x_3,x_1]^2[x_4,x_1]^2[x_3,x_2][x_3,x_4]^2,\\ 
&x_3^p [x_2,x_1]^2[x_1,x_4]^2[x_2,x_3][x_2,x_4]^2[x_3,x_4],\\
&x_4^p [x_1,x_2]^2[x_1,x_3]^2[x_1,x_4][x_3,x_2]^2[x_4,x_2],\\
&[x_1,x_2][x_1,x_4][x_3,x_4], \; [x_1,x_3][x_2,x_3][x_2,x_4]
\rangle.
\end{aligned}
\end{equation*}
\item[(c)] If $p \equiv \pm 2 \pmod{5}$, then there are two isomorphism
classes of $4$-generator UCS $p$-groups with exponent $p^2$, namely,
$G_1$ as in case~\textup{(b)} and another group $G_2$.
\end{itemize}

Moreover, $|\Aut(G_1)^{\overline{G_1}}|=20$ and 
$|\Aut(G_2)^{\overline{G_2}}|=5$.
\end{theorem}
\begin{proof}
Let $G$ be a 4-generator UCS group of exponent $p^2$, where $p$ is an
odd prime. Let $K$ denote $\Aut(G)^{\overline{G}} \leq\GL_4(p)$. 
By Theorem~\ref{UCSconstr}, $K$ is an irreducible ESQ-group. Using 
Theorem~\ref{d=4}
we immediately see that there are no UCS groups of exponent $5^2$ with
four generators. Let us use the notation introduced before Theorem~\ref{d=4}.
In particular, let $L$ denote the group $\AGL_1(5)$ acting on a 4-dimensional
vector space $V$.
By quadratic reciprocity, the number 5 is a square in the
$p$-element field if and only if $p\equiv\pm1\pmod{5}$. For these primes we have
$K=L$, for primes with $p\equiv\pm2\pmod{5}$ we can conclude that
$K\leq L$ with 5 dividing the order of $K$.

First assume that $K=L$. Since $L$ has up to equivalence a unique
faithful irreducible 4-dimensional representation, we can choose a
generating set $x_1,x_2,x_3,x_4$ of $H=H_{p,4}$
such that the matrices of generators of $L$ acting on $\overline{H}$
will be
$$
a=\begin{pmatrix}
0&1&0&0\\
0&0&1&0\\
0&0&0&1\\
-1&-1&-1&-1
\end{pmatrix},
\qquad
b=\begin{pmatrix}
0&1&0&0\\
0&0&0&1\\
1&0&0&0\\
0&0&1&0
\end{pmatrix}.
$$
Choosing the basis
$[x_1,x_2]$, $[x_1,x_3]$, $[x_1,x_4]$, $[x_2,x_3]$, $[x_2,x_4]$, $[x_3,x_4]$ in $H'$
the action of $a$ and $b$ on $H'$ is described by the matrices
$$
a'=\begin{pmatrix}
0&0&0&1&0&0\\
0&0&0&0&1&0\\
1&0&0&-1&-1&0\\
0&0&0&0&0&1\\
0&1&0&1&0&-1\\
0&0&1&0&1&1
\end{pmatrix},
\qquad
b'=\begin{pmatrix}
0&0&0&0&1&0\\
-1&0&0&0&0&0\\
0&0&0&1&0&0\\
0&0&-1&0&0&0\\
0&0&0&0&0&-1\\
0&1&0&0&0&0
\end{pmatrix}.
$$
Using the transition matrix
$$
t=\begin{pmatrix}
0&-1&1&-2&2&2\\
1&2&2&1&0&-2\\
2&0&-2&-1&-2&-1\\
-2&-2&-1&2&1&0\\
1&0&1&0&0&1\\
0&1&0&1&1&0
\end{pmatrix}
$$
of determinant $25\ne 0$ we obtain
$$
ta't^{-1}=\begin{pmatrix}
0&1&0&0&0&0\\
0&0&1&0&0&0\\
0&0&0&1&0&0\\
-1&-1&-1&-1&0&0\\
0&0&0&0&1&0\\
0&0&0&0&0&1
\end{pmatrix},
\quad
tb't^{-1}=\begin{pmatrix}
0&1&0&0&0&0\\
0&0&0&1&0&0\\
1&0&0&0&0&0\\
0&0&1&0&0&0\\
0&0&0&0&0&1\\
0&0&0&0&-1&0
\end{pmatrix}.
$$
Let $W$ denote the subspace generated by the first four rows of $t$,
and $U$ the subspace generated by the last two rows. Then
$H'=W\oplus U$, furthermore $V$ and $W$ are isomorphic
$L$-modules via the isomorphism mapping $x_i$ to the $i$th row of $t$
for $i=1,2,3,4$. 
Since $V$ is an absolutely irreducible $L$-module, the
isomorphism between $V$ and $W$ is determined up to a scalar factor.
By Theorem~\ref{3th}(iii) $G=H/N$ for some $K$-invariant normal subgroup $N$
such that $N\leq\Phi(H)$, 
$H^p\cap N=1$ and $\overline{H}$ and $H'/(N\cap H')$ are
equivalent $K$-modules. 
Therefore, there are exactly $p-1$ suitable normal subgroups
$N$, but they can be mapped
to one another using automorphisms of $H$ sending each $x_i$ to
$x_i^k$ for some fixed $k=1,\dots,p-1$. Hence there is a unique
isomorphism class of those 4-generator UCS groups $G$ of exponent~$p^2$ 
where the automorphism group of $G$ induces $L$ on $\overline{G}$.
Using the matrices above it is straightforward to write down the
defining relations of this group. (In Theorem~\ref{thm-expsquare}
we avoided inverses by reversing some commutators.)

Now consider the case, when $K$ is a proper subgroup of $L$. This can
happen only if $p\equiv\pm2\pmod{5}$. For these primes the
five-element cyclic group is an irreducible subgroup in
$\GL_4(p)$. Let $E=\mathrm{GF}(p^4)$ and $\epsilon\in E$ a
primitive fifth root of unity.
Up to conjugation $H=H_{p,4}$ has a unique
automorphism $a$ of order 5. 
Let $M=\langle a \rangle$ and let $W=[H',M]$. Then
the 6-dimensional space $H'$ decomposes as a direct sum
$H'=W\oplus\mathbf{C}_{H'}(M)$. The $M$-modules $\overline{H}$,
$H^p$, and
$W$ are isomorphic, and they can be identified with the additive group
of $E$ so that the action of $a$ becomes the multiplication by
$\epsilon$. Using this identification
every $M$-invariant normal subgroup $N\vartriangleleft H$ such that
$H/N$ is an UCS group with exponent $p^2$ has the form
$$
N_w = \{ (\vartheta,w\vartheta) \in H^p\oplus W \mid \vartheta \in E
\} \oplus \mathbf{C}_{H'}(M),
$$ 
for $w\in W\setminus\{0\}$.
Let $\gamma$ be a generator of the multiplicative group of $E$. 
The multiplication by $\gamma$ commutes with the multiplication by
$\epsilon$, hence $W$ and $\mathbf{C}_{H'}(M)$ are invariant
under the action of any automorphism of $H$ that corresponds 
to the
multiplication by $\gamma$ on $\overline{H}$.
The eigenvalues of the $\mathrm{GF}(p)$-linear transformation
determined by the multiplication by $\gamma$ on $\overline{H}\cong E$ 
are $\gamma, \gamma^p, \gamma^{p^2}, \gamma^{p^3}$, and hence its
eigenvalues on $H'$ are $\gamma^{p^i+p^j}$ for $0\leq i< j\leq 3$. Now
$\gamma^{1+p}$ has degree 4 over $\mathrm{GF}(p)$, so we infer that
the eigenvalues on $W$ are $\gamma^{(1+p)p^i}$ for $i=0,1,2,3$. 
If $p\equiv2\pmod{5}$ then let $k=(1+p)p$, while if $p\equiv3\pmod{5}$ 
then let $k=(1+p)p^2$. In both cases we have $k\equiv1\pmod{5}$.
Now the multiplication by $\gamma^k$ on $E$ has the same eigenvalues
as the action of $\gamma$ on $W$, moreover, $\epsilon^k=\epsilon$,
hence the action of $\gamma$ on $W$ is multiplication by
$\gamma^k$.
Now $\gamma^n$ transforms $N_w$ to 
\begin{equation*}
\begin{split}
&\{ (\vartheta\gamma^n,w\vartheta\gamma^{nk}) 
\in H^p\oplus W \mid \vartheta \in E
\} \oplus \mathbf{C}_{H'}(M) 
\\
&= \left\{ (\vartheta\gamma^n,\left(w\gamma^{n(k-1)}\right) \vartheta\gamma^n) 
\in H^p\oplus W \mid \vartheta \in E
\right\} \oplus \mathbf{C}_{H'}(M) = N_{w\gamma^{n(k-1)}}.
\end{split}
\end{equation*}
A simple calculation using Euclidean algorithm yields that
$\mathrm{gcd}(k-1,p^4-1)=5$. Hence $N_{w_1}$ and $N_{w_2}$ determine
isomorphic quotient groups $H/N_{w_1}$, $H/N_{w_2}$ 
provided $w_1$ and $w_2$ lie in the same
coset of the multiplicative group of $E$ modulo the fifth powers.
Field automorphisms of $E$ normalize $\langle \gamma \rangle$, hence
we can map $N_w$ to
\begin{equation*}
\begin{split}
&\{ (\vartheta^p,(w\vartheta)^p) \in H^p\oplus W \mid \vartheta \in E
\} \oplus \mathbf{C}_{H'}(M) \\
&= \{ (\vartheta^p,w^p\vartheta^p) \in H^p\oplus W \mid \vartheta \in E
\} \oplus \mathbf{C}_{H'}(M) = N_{w^p}.
\end{split}
\end{equation*}
Now it follows that for $G_1=H/N_1$ the automorphism group induces $L$
on $\overline{G_1}$. Furthermore, the other four cosets modulo the
subgroup of fifth powers are permuted cyclically by the Frobenius
automorphism, since $p$ is a primitive root modulo 5.
Hence for each $w\in E$ which is not a fifth power in $E$
the quotient groups $H/N_w$ are all isomorphic to each other, 
and this is the group
$G_2$ in Theorem~\ref{thm-expsquare}(c).
\end{proof}

It would be possible to write down explicit defining relations of $G_2$
using a generator element of the multiplicative group of
$\mathrm{GF}(p^4)$. However, we thought that a very complicated
formula will not help the reader, so we were content with stating the
existence of~$G_2$.

In Theorem~\ref{d=p}(a,c) subgroups of the affine general
linear group $\AGL_1(\F_{r^s})$ are candidates for minimal irreducible
ESQ-groups. We shall clarify when these examples arise, and
construct larger ESQ-groups.

Let $t$ be a power of a prime $r$. Identify the 1-dimensional affine
semilinear group $\AGammaL_1(\F_t)$ with the cartesian product of sets
\[
  \AGammaL_1(\F_t)=\Gal(\F_t/\F_r)\times\F_t^\times\times\F_t.
\]
Here $\Gal(\F_t/\F_r)$ denotes the group of field automorphisms of
$\F_t$. Multiplication in $\AGammaL_1(\F_t)$ is defined by
\[
  (\sigma_1,\lambda_1,\mu_1)(\sigma_2,\lambda_2,\mu_2)=
  (\sigma_1\sigma_2,(\lambda_1\sigma_2)\lambda_2,
                                   (\mu_1\sigma_2)\lambda_2+\mu_2)
\]
where $\sigma_1,\sigma_2\in\Gal(\F_t/\F_r)$,
$\lambda_1,\lambda_2\in\F_t^\times$ and $\mu_1,\mu_2\in\F_t$.
A 2-transitive action of $\AGammaL_1(\F_t)$ on $\F_t$ is given by
\[
  \alpha(\sigma,\lambda,\mu)=(\alpha\sigma)\lambda+\mu
\]
where $\alpha\in\F_t$, and $(\sigma,\lambda,\mu)\in
\Gal(\F_t/\F_r)\times\F_t^\times\times\F_t=\AGammaL_1(\F_t)$.

\begin{theorem}~\label{AGammaL}
Let $t,q$ be powers of distinct primes. Let
$L\leq\F_t^\times$, and suppose that $q\in L$, and there exist
{\em distinct} $\alpha,\beta\in L$ such that $\alpha+\beta=1$.
Define $G_L$ to be the subgroup of
$\AGammaL_1(\F_t)$ containing the elements $(\sigma,\lambda,\mu)$ such
that $\mu\in\F_t$, $\lambda\in L$, and $\sigma$ induces the identity
automorphism $\F_t^\times/L\to\F_t^\times/L$.
Then there exist $|\F_t^\times:L|$ pairwise
non-isomorphic absolutely irreducible ESQ $\F_q G_L$-modules of
dimension $|L|$. In particular, $\AGammaL_1(\F_t)$ is an absolutely
irreducible ESQ-subgroup of $\GL_{t-1}(\F_q)$ if $t>3$ and $\gcd(t,q)=1$.
\end{theorem}

\begin{proof}
Let $V=(\F_q)^t$ be the permutation module for $\AGammaL_1(\E)$ where
$\E:=\F_t$ has (prime) characteristic $r$. Let $(e_\alpha)_{\alpha\in\E}$ be a
basis for $V$ indexed by $\alpha\in\E$. The action of $\AGammaL_1(\E)$
on $V$ is given by 
\[
  e_\alpha(\sigma,\lambda,\mu)=e_{(\alpha\sigma)\lambda+\mu}
    \qquad\qquad(\alpha\in\E).
\]
Our proof has two cases. Assume first that $q\equiv 1\pmod r$.
In this case the hypothesis $q\in L$ holds trivially, as $1\in L$ and
$q=1$ in $\E$. We show later that the second case when $q\not\equiv
1\pmod r$, reduces to this first case. 

Let $T\colon\E\to\F_r$ denote the
absolute trace function: $T(\alpha)=\sum_\sigma
\alpha\sigma$ where $\sigma$ ranges over $\Gal(\E/\F_r)$.
Let $\zeta\in\F_q^\times$ have order $r$, and define
\[
  f_\alpha=\sum_{\mu\in\E} \zeta^{T(\alpha\mu)}e_\mu.
\]
Since $|\E|^{-1}\sum_{\alpha\in\E} \zeta^{T(\alpha(\mu-\nu))}=0$
if $\mu\ne\nu$, and 1 otherwise, we have
\[
  |\E|^{-1}\sum_{\alpha\in\E} \zeta^{-T(\nu\alpha)}f_\alpha
    =\sum_{\mu\in\E}\left(|\E|^{-1}\sum_{\alpha\in\E} 
      \zeta^{T(\alpha(\mu-\nu))}\right)e_\mu
    =e_\nu.
\]
Therefore $(f_\alpha)_{\alpha\in\E}$, defines a new basis for $V$.
The action of $\AGammaL_1(\E)$ on the new basis is monomial, and given by
\begin{equation}\label{e4}
  f_\alpha(\sigma,\lambda,\mu)=
    \zeta^{-T((\alpha\sigma)\lambda^{-1}\mu)}f_{(\alpha\sigma)\lambda^{-1}}.
\end{equation}

Consider the normal subgroup $N=\{(1,1,\mu)\mid \mu\in\E\}$ of
$G_L$. By equation (\ref{e4}),
$f_\alpha(1,1,\mu)=\zeta^{-T(\alpha\mu)}f_{\alpha}$. Thus each
$\langle f_\alpha\rangle$ is an irreducible $\F_qN$-module.
The non-degeneracy of the map
$\E\times\E\to\F_r\colon (\alpha,\beta)\mapsto T(\alpha\beta)$
implies that there is an $N$-module isomorphism
\begin{equation}\label{e5}
  \langle f_\alpha\rangle\cong\langle f_\beta\rangle
    \quad\text{if and only if}\quad \alpha=\beta.
\end{equation}

Let $\lambda'L$ be a coset of $L$ in $\E^\times$, and set
\[
  W(\lambda'L)=\langle f_{\alpha}\mid\alpha\in\lambda'L\rangle
  =\sum_{\alpha\in\lambda'L}\langle f_{\alpha}\rangle.
\]
It follows from equation (\ref{e4}) that $W(\lambda'L)$ is a $G_L$-module.
We shall show that it is irreducible.
Set $M=\{(1,\lambda,\mu)\mid\lambda\in L,\mu\in\E\}$.
Then $M$ is a normal subgroup of $G_L$, and by (\ref{e5}) the inertia subgroup of
$\langle f_{\lambda'}\rangle$ in $M$ is $N$. Hence by Clifford's
theorem, the induced module
$\text{Ind}_N^M(\langle f_{\lambda'}\rangle)=W(\lambda'L)$ is
$M$-irreducible, and {\em a fortiori} $G_L$-irreducible.

There are two decompositions of $V$:
\[
  V=\sum_{\alpha\in\E}\langle f_\alpha\rangle\quad\text{and}\quad
  V=\langle f_0\rangle\oplus\sum_{\lambda'L\in\E^\times/L}W(\lambda'L).
\]
The first is as a direct sum of irreducible $\F_q N$-modules, and the
second, a direct sum of irreducible $\F_q G_L$-modules.
If $\alpha',\beta'\in\lambda'L$, then by equation (\ref{e4}) 
\begin{align*}
  (f_{\alpha'}\wedge f_{\beta'})(\sigma,\lambda,\mu)
    =\zeta^{-T((\alpha'+\beta')\sigma\lambda^{-1}\mu)}
    f_{(\alpha'\sigma)\lambda^{-1}}\wedge f_{(\beta'\sigma)\lambda^{-1}}.
\end{align*}
Hence if $\alpha'\ne\beta'$, then $\langle f_{\alpha'}\wedge
f_{\beta'}\rangle\cong\langle f_{\alpha'+\beta'}\rangle$ as $\F_q
N$-modules by (\ref{e4}). 
Thus $W(\lambda'L)$ is an ESQ-module if
and only if $\alpha'+\beta'\in\lambda'L$ for distinct 
$\alpha',\beta'\in\lambda'L$. Equivalently, $\alpha+\beta=1$
where $\alpha=\alpha'/(\alpha'+\beta')$ and $\beta=\beta'/(\alpha'+\beta')$
in $L$ are distinct. Clearly
$W(\lambda'L)\cong W(\lambda''L)$ as $N$-modules if
and only if $\lambda'L=\lambda''L$. Therefore $W(\lambda'L)\cong
W(\lambda''L)$ as $G_L$-modules if and only if $\lambda'L=\lambda''L$.
In summary, the $W(\lambda'L)$ provide $|\E^\times:L|$ pairwise
non-isomorphic absolutely irreducible ESQ $\F_q G_L$-modules of
dimension $|L|$.

Suppose now that $q\not\equiv 1\pmod r$, and $q\in L$ holds.
We temporarily enlarge the field of scalars from $\F_q$
to the finite field $\F_q(\zeta)$, where $\zeta$ has order
$r$. View $W(\lambda'L)$ as an absolutely irreducible ESQ
$\F_q(\zeta) G_L$-module. We shall show that there exists an $\F_q G_L$-module 
$U(\lambda'L)$ which satisfies $W(\lambda'L)\cong
U(\lambda'L)\otimes_{\F_q} \F_q(\zeta)$. It then follows that the
$U(\lambda'L)$ are pairwise non-isomorphic absolutely irreducible ESQ
$\F_q G_L$-modules of dimension $|L|$. By a theorem of Brauer
(see \cite[Theorem~VII.1.17]{HB} and \cite{GH}), the module $U(\lambda'L)$
exists if and only if the character $\chi$ of $W(\lambda'L)$ has
values in $\F_q$. We shall show that $\chi(\sigma,\lambda,\mu)\in\F_q$
for all $(\sigma,\lambda,\mu)\in G_L$ by showing
$\chi(\sigma,\lambda,\mu)^q=\chi(\sigma,\lambda,\mu)$. By equation (\ref{e4})
\[
  \chi(\sigma,\lambda,\mu)=
  \sum_{\{\alpha\in\lambda'L\,\mid\,(\alpha\sigma)\lambda^{-1}=\alpha\}}
  \zeta^{-T((\alpha\sigma)\lambda^{-1}\mu)}=
  \sum_{\{\alpha\in\lambda'L\,\mid\,(\alpha\sigma)\alpha^{-1}=\lambda\}}
  \zeta^{-T(\alpha\mu)}.
\]
However, 
$\{\alpha\in\lambda'L\mid(\alpha\sigma)\alpha^{-1}=\lambda\}=
\{\alpha q\in\lambda'L\mid((\alpha q)\sigma)(\alpha q)^{-1}=\lambda\}$
as $q\in L$ is fixed by $\sigma$. Therefore
$\chi(\sigma,\lambda,\mu)=\chi(\sigma,\lambda,\mu)^q$ as desired.

If $L=\E^\times$, then $G_L=\AGammaL_1(\E)$. Moreover,
$q\in\E^\times$ holds as $t,q$ are powers of distinct primes. If
$t>r$, then take 
$\alpha$ to be an element of $\E^\times$ not in $\F_r^\times$, and
so that $\beta=1-\alpha$ satisfies $\alpha+\beta=1$ and
$\alpha\ne\beta$. If $t=r$, then take $\alpha=2$ and $\beta=-1$.
Clearly $\alpha+\beta=1$ and $\alpha\ne\beta$
provided $r>3$. This proves that $\AGammaL_1(\E)$ is an absolutely
irreducible ESQ-subgroup of $\GL_{t-1}(\F_q)$ if $t>3$ and $\gcd(t,q)=1$.
\end{proof}

We next prove Theorem~\ref{main}(e).

\begin{proof}[Proof of Theorem~$\ref{main}$(e)]
Consider parts (i) and (iii). Let $p$ be an odd prime, and let $q=p^k$.
Let $V=(\F_q)^3$ be the natural $\SO_3(q)$-module.
Choose a basis $x_1,x_2,x_3$ for $V$, and the basis
$x_2\wedge x_3,x_3\wedge x_1,x_1\wedge x_2$ for $\VwedgeV$. The matrix
of $g\wedge g$ is $\det(g)(g^{-1})^\T$. As $\det(g)=1$ and
$g^\T g=I$, it follows that $g\wedge g=g$ and so $V$
is an irreducible ESQ $\SO_3(q)$-module. By Theorem~\ref{UCSconstr}(b)
there exists an exponent-$p^2$ UCS $p$-group of order~$q^6$.
This proves part~(i). Similarly, by the last sentence of
Theorem~\ref{AGammaL}, $\AGammaL_1(8)$ is an absolutely irreducible
ESQ-subgroup of $\GL_7(\F_q)$ for odd $q$. Part~(iii) now follows by
Theorem~\ref{UCSconstr}(b).

Consider part~(ii). Parts (a) and (b) of Theorem~\ref{UCSconstr} are
true with $k=1$. Thus if $G$ is an exponent-$p^2$ UCS-group of order $p^{10}$,
then $\Aut(G)^{\overline G}$ is an irreducible ESQ-subgroup of $\GL_5(p)$.
It follows from
Theorem~\ref{d=5} (with $q=p$) that $p^5\equiv1\pmod{11}$. (Additionally,
$|\Aut(G)^{\overline G}|$ is divisible by $11$.) Conversely,
if $p^5\equiv1\pmod{11}$, or more generally if $q=p^k$ satisfies
$q^5\equiv1\pmod{11}$, then there exists an exponent-$p^2$ UCS-group
of order $q^{10}$ by Theorems~\ref{d=5} and \ref{UCSconstr}.
This proves part~(ii). 
\end{proof}

Note that $q^{12}=(q^2)^6$ and so by Theorem~\ref{main}(e)(i) there
exist UCS-groups of order $q^{12}$ and exponent $p^2$ for all
powers $q$ of an odd prime $p$. 

\section*{Acknowledgements}

We are grateful to Eamonn O'Brien for making available to us a {\sc Magma}
\cite{Magma}
database of class-2 $p$-groups with 4 generators and exponent $p$; and to
P\'al Heged\H us for his valuable comments.
The second and the third authors were supported by the Hungarian Scientific 
Research Fund (OTKA) grants NK725223 and~NK72845. Much of the research 
by the third author was carried out 
in the Computer and Automation Research Institute of the Hungarian 
Academy of Sciences (MTA SZTAKI). In addition he was supported by the 
research grant PTDC/MAT/101993/2008 of the 
{\em Funda\c c\~ao para a Ci\^ encia e a Tecnologia} (Portugal).

\frenchspacing
\newcommand{\etalchar}[1]{$^{#1}$}
\def\cprime{$'$}

\end{document}